\documentclass[reqno]{amsart}

\usepackage{amsrefs}
\usepackage{amssymb}
\usepackage{enumerate}
\usepackage{tikz-cd}
\usepackage{etoolbox}
\usepackage{microtype}
\usepackage{lmodern}
\usepackage{url}
\usepackage{hyperref}

\theoremstyle{plain}
\newtheorem{thm}{Theorem}[subsection]
\newtheorem{defn}[thm]{Definition}
\newtheorem{que}[thm]{Question}
\newtheorem{rem}[thm]{Remark}
\newtheorem{prop}[thm]{Proposition}
\newtheorem{lem}[thm]{Lemma}
\newtheorem{cor}[thm]{Corollary}
\newtheorem{exa}[thm]{Example}
\newtheorem{conj}[thm]{Conjecture}

\theoremstyle{plain}
\newtheorem*{thm*}{Theorem}
\newtheorem*{defn*}{Definition}
\newtheorem*{prop*}{Proposition}
\newtheorem*{lem*}{Lemma}

\def\cal{\mathcal}

\def\cA{{\cal A}}

\def\cD{{\cal D}}

\def\cF{{\cal F}}

\def\cU{{\cal U}}

\def\a{\alpha}
\def\b{\beta}

\def\d{\delta}
\def\ve{\varepsilon}
\def\f{\phi}
\def\k{\kappa}
\def\l{\lambda}
\def\L{\Lambda}

\def\r{\rho}
\def\s{\sigma}
\def\ta{\theta}
\def\t{\tau}
\def\w{\omega}

\def\O{\Omega}

\def\del{\partial}
\def\adel{\ol{\partial}}
\def\D{\Delta}
\def\G{\Gamma}
\def\io{\iota}

\def\p{\phi}

\def\m{\mu}
\def\e{\eta}

\def\bC{{\mathbb C}}
\def\bN{{\mathbb N}}

\def\bR{{\mathbb R}}

\def\ol{\overline}

\def\*{\ast}
\def\wdg{\wedge}

\def\isom{\cong}
\def\tens{\otimes}
\def\dsum{\oplus}

\def\<{\langle}
\def\>{\rangle}

\def\mto{\mapsto}
\def\to{\rightarrow}

\DeclareMathOperator{\vol}{vol}

\DeclareMathOperator{\flip}{flip}
\DeclareMathOperator{\proj}{proj}
\DeclareMathOperator{\Hom}{HOM}

\DeclareMathOperator{\id}{id}

\title{Generalized symmetry in noncommutative (complex) geometry}
\author{Suvrajit Bhattacharjee}
\address{Stat-Math Unit\\ Indian Statistical Institute\\ 203, B.T. Road\\ Kolkata-700108}
\email{suvra.bh@gmail.com}
\author{Indranil Biswas}
\address{School of Mathematics\\ Tata Institute of Fundamental Research\\ Homi Bhabha Road\\ Navy Nagar\\ Colaba\\ Mumbai-400005}
\email{indranil@math.tifr.res.in}
\author{Debashish Goswami}
\address{Stat-Math Unit\\ Indian Statistical Institute\\ 203, B.T. Road\\ Kolkata-700108}
\email{goswamid@isical.ac.in}
\numberwithin{equation}{section}

\begin{document}

\begin{abstract}
We introduce Hopf algebroid covariance on Woronowicz's differential calculus. Using it, we develop quite a general framework of noncommutative complex geometry that subsumes the one in \cite{MR3720811}. We present transverse complex and K\"ahler structures as examples and discuss several other examples. Relation with past literature is described. 
\end{abstract}

\maketitle

\section{Introduction}

Symmetry plays an important, often decisive, role in almost all areas of mathematics; especially in geometry and topology. Classically, symmetry means group action on spaces. However, one has to pass from groups to groupoids to capture local symmetry in an efficient way. For example, it is more natural to consider the (Lie)groupoid of isometries of a Riemannian manifold which is not globally symmetric or homogeneous. The natural domain for the characteristic classes of certain geometric structures are in fact the cohomology of the classifying spaces of (Lie)groupoids. One can go further, saying that groupoids provide a concept of generalized symmetry that is essential, as exemplified above and spectacularly apparent in the theory of foliations. In the realm of noncommutative geometry, symmetry is captured by action or coaction of Hopf algebras on (co)algebras, which is the noncommutative version of a space.

\medskip\noindent

The concept of Hopf algebroids generalizes that of groupoids, providing a way of considering generalized symmetry in noncommutative geometry. They can be thought of as Hopf algebras over noncommutative bases. Initially conceived by algebraic topologists, Hopf algebroids over commutative bases have been used extensively in geometry and topology. Problems appear when one considers generalizing the definition to noncommutative bases. The pre-Hopf algebroid level, i.e., the definition of a bialgebroid is usually accepted as the correct generalization of a bialgebra over a noncommutative base. The problem is with the addition of an antipode. A description of the various definitions is given in the introduction of \cite{MR2817646}. The definition used in this article also comes from that paper which was first given in \cite{MR2043373}.

\medskip\noindent

The theory of noncommutative complex geometry was initiated in \citelist{\cite{MR3428362}\cite{MR3720811}}, although there are precursors, see \citelist{\cite{ MR3073899}\cite{ MR1614993}\cite{ MR2773332}\cite{MR1977884}}. It attempts to provide a fresh insight into various aspects of noncommutative geometry, such as the construction of spectral triples for quantum groups, by considering ``complex structures". It also promises a fruitful interaction between noncommutative geometry and noncommutative projective algebraic geometry. Identifying ``differential forms" as the basic objects of study, the framework of noncommutative complex geometry is developed in the setting of Woronowicz's differential calculus, see \cite{ MR994499}. The classical complex geometry being the obvious example, the set-up in \cite{MR3720811} takes as its motivating example the family of quantum flag manifolds. One can go, as shown in there, as far as proving a version of Hard Lefschetz theorem.

\medskip\noindent

Singular spaces, such as the leaf space of a foliation, have been studied extensively in classical geometry as well as noncommutative geometry. These spaces provided the main impetus for the development of noncommutative geometry, see \cite{ MR679730}. Classically, ``transverse geometry" attempts to study such singular spaces using symmetry, which most of the time turns out to be a pseudogroup. This was exemplified in the beautiful paper \cite{MR614370}. It led to the systematic study of spaces with pseudogroup symmetry. It is natural to ask whether one can do complex geometry over such spaces. That one can, was done in a volume of works, \citelist{\cite{MR1151583}\cite{MR1042454}\cite{MR1042454}}, to name a few.

\medskip\noindent

Now, pseudogroups and groupoids are very much noncommutative in their nature. This led to Connes' construction of the highly noncommutative groupoid $C^*$-algebra of the holonomy groupoid of a foliation, which is successfully applied to the questions in index theory. However, the fact that groupoids consist of symmetries is lost in this construction. To take the symmetry into account, one is naturally led to the language of Hopf algebroids, as shown in \citelist{\cite{MR2853081}\cite{MR1722796}\cite{MR2311185}}.

\medskip\noindent

Thus, the study of complex geometry over such singular spaces consists of studying regular spaces with highly noncommutative symmetry, which are also generalized, in that they are not Hopf algebras. Let us mention a few of our results to exemplify what we have said above. For the convenience of the reader, let us recall the definition of a foliation first.

\begin{defn*}
	A foliation $\cF$ on $M$ is given by a cocycle $\cU=\{U_i,f_i,g_{ij}\}$ modeled on a manifold $N_0$ ($\bR^n$ or $\bC^n$), i.e.,
	\begin{enumerate}[i)]
	\item $\{U_i\}$ is an open covering of $M$;
	\item $f_i : U_i \to N_0$ are submersions with connected fibers defining $\cF$;
	\item $g_{ij}$ are local diffeomorphisms of $N_0$ and $g_{ij}f_j=f_i$ on $U_i \cap U_j$.
	\end{enumerate}
\end{defn*}
	
The manifold $N=\sqcup f_i(U_i)$ is called the transverse manifold of $\cF$ associated to
the cocycle $\cU$, and the pseudogroup $P$ generated by $g_{ij}$ is called the holonomy pseudogroup on the transverse manifold. To such a pseudogroup $P$, one can associate an \'etale (effective) groupoid $\G(P)$ over $N$, called the holonomy groupoid. Then a ``geometric structure'' on the foliated manifold $(M,\cF)$ is a ``structure" on the manifold $N$ equivariant under the action of the groupoid $\G(P)$. Now, it is known that (see Subsection \ref{subsec 4.2.2}) the space $C_c^{\infty}(\G(P))$ of compactly supported smooth functions on $\G(P)$ is a Hopf algebroid in the sense of Definition \ref{defn 4.2.8}. We need a conversion such as 
\[
	\G(P) \text{-equivariant structures on } N \equiv C_c^{\infty}(\G(P)) \text{-covariant structures on } C^{\infty}(N).
\]

The first candidate of such a structure is the differential forms. The translation in this case is more or less routine but what is hard is to identify the property that relates the groupoid action and the exterior derivative. One can see the difficulty from the fact that groupoids generalize both spaces and groups. In the case of spaces, the derivative satisfies a Leibniz rule and in the group case, it commutes with the action. Our first nontrivial result is the translation of this property to the Hopf algebroid language, more precisely, 

\begin{lem*}
	The differential $d$ on $N$ satisfies 
	\begin{equation}
	d(a \cdot \w)=d(\ve_l(a)) \wdg \w + a \cdot d(\w)
	\end{equation} for $a \in C_c^{\infty}(\G(P))$ and $\w \in \O(N)$. 
\end{lem*}

Denote by $\O(N)^{\G(P)}$ the $\G(P)$-invariant forms. Then forms on the ``orbit or leaf space" are captured as follows.

\begin{prop*}
The pair
$(\O(N),\,d)$ is a $C_c^{\infty}(\G(P))$-covariant differential calculus, and we have \[(\O(N)^{\G(P)},\,d)=(\O(N)_{C_c^{\infty}(\G(P))},\,d)\] as differential graded algebras, where the right-hand side denotes the invariant subalgebra under the action of the Hopf algebroid $C_c^{\infty}(\G(P))$.
\end{prop*}

The framework in which to talk about $C_c^{\infty}(\G(P))$-covariant differential calculus and such statements is our principal achievement in this paper. In fact, the goal of the present article is to introduce Hopf algebroid symmetry in noncommutative geometry and build such a framework. We formulate and study a quite general framework of Hopf algebroid-covariance of noncommutative complex and K\"ahler structures. We have been able to accommodate all the existing examples in our framework. Another notable and novel aspect of our work is a new definition of Hopf algebroid action or covariance on differential calculus which seems to work in a very general context. We present the Connes-Moscovici Hopf algebroid as one of the most interesting examples of our setup. Let us briefly describe the plan of the paper. In Section \ref{sec 4.2}, we recall the definition of a Hopf algebroid. Foliations and \'etale groupoids in general, are also discussed in some detail and is shown to provide examples. In the last subsection, we recall the already-existing Connes-Moscovici left bialgebroid and we prove that it is a Hopf algebroid by putting a right bialgebroid structure. Section \ref{sec 3} introduces modules and module-algebras over a Hopf algebroid. In the final subsection of this section, we put $\*$-structures on Hopf algebroids which is essential in order to view them as symmetry objects in noncommutative geometry. In Section \ref{sec 4.3}, our framework and applications to noncommutative K\"ahler structures are described. Hopf algebroid covariance is introduced. The necessary modifications of the framework in \cite{MR3720811} are described and along the way, examples coming from foliations are provided.  Finally, Section \ref{sec 4.6} discusses some future directions. We could prove a version of the Hodge decomposition theorem and the formality theorem, which generalizes the case for foliations but we decided to omit these two from the present article, owing to its length.

\medskip\noindent

Before ending this introduction, we would like to mention that it is not our intention to bring in Hopf algebroids and investigate such a general situation just for generality's sake. Even in the classical case, one cannot always avoid having symmetries encoded in a more general object than a group. One needs a framework to work with these situations, and the theory of foliations amply exemplify such a need. We are after such a framework where one can take the naturally-occurring Hopf algebroid symmetry into account. If one needs further justification of where indeed these objects occur, we refer to subfactor theory. The knowledgeable reader might immediately recall \citelist{\cite{MR1745634}\cite{MR1800792}\cite{MR1913440}} where Hopf algebroids appear in the form of weak Hopf algebras as the ``Galois group of automorphisms.'' If one forgets the analysis, there is no way one can escape Hopf algebroids as is wonderfully explained in \cite{MR2004729}. Further pointers are given at the end of this article (see \ref{pointer}).  

\section{Preliminaries}\label{sec 4.2}

\subsection{Hopf algebras over noncommutative base - Hopf algebroids} \label{subsec 4.2.1}

We recall the definition of Hopf algebroids from \cite{MR2817646}. See also \cites{MR2553659, MR2043373}. We begin by defining a generalization of bialgebras.

\begin{defn}\label{defn 4.2.1}
Let $A$ be a $\bC$-algebra. An $(s,t)$-ring over $A$ is a $\bC$-algebra $H$ with homomorphisms $s : A \to H$ and $t : A ^{op} \to H$ whose images commute in $H$.
\end{defn}

The functions $s$ and $t$ are referred to as the source and target maps respectively. An $(s,t)$-ring structure is equivalent to the structure of an $A^e$-algebra on $H$.

\begin{defn}\label{defn 4.2.2}
Let $H$ be an $(s,t)$-ring over $A$. The Takeuchi product is the subspace
\[H \times_A H:=\{\sum_i h_i \tens_A h'_i \in H \tens_A H \mid \sum_i h_it(a) \tens h'_i=\sum_i h_i \tens h'_is(a) \quad \forall a \in A\}
\]
of $H \tens_A H$, where the tensor product $\tens_A$ is defined with respect to the
following $(A,A)$-bimodule structure on $H$:
\begin{equation}\label{eq 4.2.1}
a_1 \cdot h \cdot a_2:=s(a_1)t(a_2)h, \qquad a_1,a_2 \in A, \quad h \in H.
\end{equation}
\end{defn}

This Takeuchi product becomes a unital algebra with factorwise multiplication as well as an 
$(s,t)$-ring. Before we go onto the definition of a bialgebroid, let us recall the definition of an $A$-coalgebra.

\begin{defn}
Let $A$ be a $\bC$-algebra. A coalgebra over $A$ is a triple $(C,\D,\ve)$ with $C$ an $(A,A)$-bimodule, $\D$ an $(A,A)$-bimodule morphism called the comultiplication, $\ve : C \to A$ an $(A,A)$-bimodule morphism called the counit, and such that 
\begin{equation}\label{eq 1.1.2}
(\D \tens_A \id)\D=(\id \tens_A \D)\D, \qquad (\id \tens_A \ve)\D=(\ve \tens_A \id)\D=\id.
\end{equation}
\end{defn}

A left bialgebroid over $A$ is then an algebra with a compatible coalgebra structure over $A$. More precisely,

\begin{defn}\label{defn 4.2.3}
Let $A_l$ be a $\bC$-algebra. A left bialgebroid over $A_l$ is an $(s_l,t_l)$-ring $H_l$ equipped with the structure of an $A_l$-coalgebra $(\D_l,\ve_l)$ with respect to the $(A_l,A_l)$-bimodule structure \eqref{eq 4.2.1}, subject to the following conditions:
\begin{enumerate}[i)]
\item the (left) coproduct $\D_l : H_l \to H_l \tens_{A_l} H_l$ maps into
the subset $H_l \times_{A_l} H_l$ and defines a morphism $\D_l : H_l \to H_l \times_{A_l} H_l$ of unital $\bC$-algebras;
\item the (left) counit has the property: 
\begin{equation}\label{eq 4.2.2}
\ve_l(hh')=\ve_l(hs_l(\ve_lh'))=\ve_l(ht_l(\ve_lh')) \qquad h,h' \in H_l.
\end{equation}
\end{enumerate}
\end{defn}

We denote the above left bialgebroid by $(H_l,A_l,s_l,t_l,\D_l,\ve_l)$ or simply by $H_l$.

\begin{rem}\label{rem 4.2.4}
From \eqref{eq 4.2.2} above and the fact that $\ve_l$ is an $(A_l,A_l)$-bimodule morphism, it follows
that $\ve_l(s_l(a)h)=a\ve_l(h)$, $\ve_l(t_l(a)h)=\ve_l(h)a$, and it also follows
that $\ve_l(1_{H_l})=1_{A_l}$. So we have that $\ve_ls_l=\ve_lt_l=\id_{A_l}$.
\end{rem}

\begin{lem}\label{lem 4.2.5} 
In a left bialgebroid, the left counit is unique. 
\end{lem}

\begin{proof}
Indeed, if both $\ve_l^1$ and $\ve_l^2$ make $(H_l,A_l,s_l,t_l,\D_l,\ve_l^1)$ and
$(H_l,A_l,s_l,t_l,\D_l,\ve_l^2)$ left bialgebroids, then we have:
\[\ve_l^2(h)=\ve_l^2(s_l\ve_l^1(h_1)h_2)=\ve_l^1(h_1)\ve_l^2(h_2)=\ve_l^1(t_l\ve_l^2(h_2)h_1)=\ve_l^1(h).\]
\end{proof}

Given an $(s,t)$-ring $H$, there is another $(A,A)$-bimodule structure on $H$:
\begin{equation}\label{eq 4.2.3}
a_1 \cdot h \cdot a_2=ht(a_1)s(a_2), \qquad a_1,a_2 \in A \quad h \in H.
\end{equation}

With respect to this bimodule structure, the tensor product $\tens_A$ is defined. Inside $H \tens_A H$, there is the Takeuchi product: \[H \times^{A} H:=\{\sum_i h_i \tens_A h'_i \in H \tens_A H \mid \sum_i s(a)h_i \tens h'_i=\sum_i h_i \tens t(a)h'_i \quad \forall a \in A\}.
\] This again becomes a unital algebra with factorwise multiplication and also is an $(s,t)$-ring.

\begin{defn}\label{defn 4.2.6}
Let $A_r$ be a $\bC$-algebra. A right bialgebroid over $A_r$ is an $(s_r,t_r)$-ring $H_r$ equipped with the structure of an $A_r$-coalgebra $(\D_r,\ve_r)$ with respect to the $(A_r,A_r)$-bimodule structure \eqref{eq 4.2.3}, subject to the following conditions:
\begin{enumerate}[i)]
\item the (right) coproduct $\D_r : H_r \to H_r \tens_{A_r} H_r$ maps into $H_r \times^{A_r} H_r$ and defines a morphism $\D_r : H_r \to H_r \times^{A_r} H_r$ of unital $\bC$-algebras;
\item the (right) counit has the property: 
\begin{equation}\label{eq4.2.4}
\ve_r(hh')=\ve_r(s_r(\ve_rh)h')=\ve_r(t_r(\ve_rh)h') \qquad h,h' \in H_r.
\end{equation}
\end{enumerate}
\end{defn}

We denote a right bialgebroid by $(H_r,A_r,s_r,t_r,\D_r,\ve_r)$ or simply by $H_r$. Note that if $(H_l,A_l,s_l,t_l,$ $\D_l,\ve_l)$ is a left bialgebroid, then $(H_l^{op},A_l,t_l,s_l,\D_l,\ve_l)$ is a right bialgebroid.

\begin{rem}\label{rem 4.2.7}
As in Remark \ref{rem 4.2.4}, we have $\ve_rs_r=\ve_rt_r=\id_{A_r}$. Also as above, the right counit is unique.
\end{rem}

\subsubsection*{\bf Sweedler notation} We shall use Sweedler notation with subscripts $\D_l(h)=h_{(1)} \tens h_{(2)}$ for left 
comultiplication while the right comultiplication are indicated by superscripts: $\D_r(h)=h^{(1)} \tens h^{(2)}.$

\medskip\noindent

We now define a Hopf algebroid as an algebra endowed with a left and a right bialgebroid structure together 
with an antipode ``intertwining'' the left bialgebroid and the right bialgebroid structures. More precisely:

\begin{defn}\label{defn 4.2.8}
A Hopf algebroid is given by a triple $(H_l,H_r,S)$, where $H_l=(H_l,A_l,s_l,t_l,\D_l,\ve_l)$ is a left $A_l$-bialgebroid and $H_r=(H_r,A_r,s_r,t_r,\D_r,\ve_r)$ is a right $A_r$-bialgebroid on the same $\bC$-algebra $H$, and $S :H \to H$ is invertible $\bC$-linear. These structures are
subject to the following four conditions:
\begin{enumerate}[i)]
\item the images of $s_l$ and $t_r$ as well as those of $t_l$ and $s_r$, coincide:
\begin{equation}\label{eq 4.2.5}
s_l\ve_lt_r=t_r, \quad t_l\ve_ls_r=s_r, \quad s_r\ve_rt_l=t_l, \quad t_r\ve_rs_l=s_l;
\end{equation}
\item mixed coassociativity holds:
\begin{equation}\label{eq 4.2.6}
(\D_l \tens \id_H)\D_r=(\id_H \tens \D_r)\D_l, \qquad (\D_r \tens \id_H)\D_l=(\id_H \tens \D_l)\D_r;
\end{equation}
\item for all $a_1 \in A_l$, $a_2 \in A_r$ and $h \in H$, we have
\begin{equation}\label{eq 4.2.7}
S(t_l(a_1)ht_r(a_2))=s_r(a_2)S(h)s_l(a_1);
\end{equation}
\item the antipode axioms hold:
\begin{equation}\label{eq 4.2.8}
\m_H(S \tens \id_H)\D_l=s_r\ve_r, \qquad \m_H(\id_H \tens S)\D_r=s_l\ve_l.
\end{equation}
\end{enumerate}
\end{defn}

We apply $\ve_r$ to the first two and $\ve_l$ to the second pair of identities in \eqref{eq 4.2.5} and get that $A_l$ and $A_r$ are anti-isomorphic as $\bC$-algebras:
\begin{equation}\label{eq 4.2.9}
\begin{split}
\p:=\ve_rs_l : A_l^{op} \to A_r, \qquad \p^{-1}:=\ve_lt_r : A_r \to A_l^{op},\\
\ta:=\ve_rt_l : A_l \to A_r^{op}, \qquad \ta^{-1}:=\ve_ls_r : A_r^{op} \to A_l.
\end{split}
\end{equation}

The antipode is anti-algebra and anti-coalgebra morphism (between different coalgebras) and satisfies 
the equations
\begin{equation}\label{eq 4.2.10}
\flip  (S \tens S)\D_l=\D_rS, \qquad \flip  (S \tens S)\D_r=\D_lS,
\end{equation}
where $\flip : H \tens_{\bC} H \to H \tens_{\bC} H$ is the flip permuting two factors of the tensor product
(this becomes an $(A_l,A_l)$-respectively $(A_r,A_r)$-bimodule). Similar formulas hold for the inverse $S^{-1}$. The following identities
will be used:
\begin{equation}\label{eq 4.2.11}
\begin{matrix}
s_r\ve_rs_l=Ss_l, & s_l\ve_ls_r=Ss_r, & s_r\ve_rt_l=S^{-1}s_l, & s_l\ve_lt_r=S^{-1}s_r,\\
t_r\ve_rs_l=St_l, & t_l\ve_ls_r=St_r, & t_r\ve_rt_l=S^{-1}t_l, & t_l\ve_lt_r=S^{-1}t_r,\\
\ve_rs_l\ve_l=\ve_rS, & \ve_ls_r\ve_r=\ve_lS, & \ve_rt_l\ve_l=\ve_rS^{-1}, & \ve_lt_r\ve_r=\ve_lS^{-1}, 
\end{matrix}
\end{equation}
and
\begin{equation}\label{eq 4.2.12}
\begin{matrix}
\m_H(S \tens s_l\ve_l)\D_l=S, & \m_H(s_r\ve_r \tens S)\D_r=S,\\
\m_{H^{op}}(\id_H \tens S^{-1})\D_l=t_r\ve_r, & \m_{H^{op}}(S^{-1} \tens \id_H)\D_r=t_l\ve_l,\\
\m_{H^{op}}(t_l\ve_l \tens S^{-1})\D_l=S^{-1}, & \m_{H^{op}}(S^{-1} \tens t_r\ve_r)\D_r=S^{-1}.
\end{matrix}
\end{equation}

\begin{lem}\label{lem 4.2.9}
In a Hopf algebroid, the antipode is unique. 
\end{lem}

\begin{proof}
Indeed, if both $S_1$ and $S_2$ make $(H_l, H_r, S_1)$ and $(H_l, H_r, S_2)$ Hopf algebroids then
we have
\[
\begin{aligned}
S_2(h)=s_r\ve_r(h^{(1)})S_2(h^{(2)})&=S_1(h^{(1)}_{(1)})h^{(1)}_{(2)}S_2(h^{(2)})\\
&=S_1(h_{(1)})h_{(2)}^{(1)}S_2(h_{(2)}^{(2)})=S_1(h_{(1)})s_l\ve_l(h_{(2)})=S_1(h).
\end{aligned}    
\] 
\end{proof}

Finally, note that if $(H_l,H_r,S)$ is a Hopf algebroid, then $(H_r^{op}, H_l^{op}, S^{-1})$ is also a Hopf algebroid.

\subsection{The main example - \'Etale groupoids}\label{subsec 4.2.2} We now introduce our main example besides Hopf algebras. A Hopf algebra is a Hopf algebroid with $A_l=A_r=\bC$. We follow \cite{MR2012261}. See also \cites{MR1303779, MR3448330, MR2853081}.

\begin{defn}\label{defn 4.2.10}
A groupoid $G$ is a small category in which each arrow is invertible. More explicitly, a groupoid consists of a
space of objects $G_0$, a space of arrows $G_1$ (often denoted by $G$ itself) and five structure maps relating the two:
\begin{enumerate}[i)]
\item source and target maps $s,t : G_1 \to G_0$, assigning to each arrow $g$ its source $s(g)$ and target $t(g)$; one says that $g$ is from $s(g)$ to $t(g)$;
\item a partially defined composition of arrows, that is, only for those arrows $g,h$ for which source and target match, that is $s(g)=t(h)$; in other words, a map $m : G_2:=G_1 ~^s\times^t_{G_0} G_1 \to G_1$, $(g,h) \mto gh$ that is associative whenever defined, producing the composite arrow going from $s(gh)=s(h)$ to $t(gh)=t(g)$;
\item a unit map $1 : G_0 \to G_1$, $x \mto 1_x$, that has the property $1_{t(g)}g=g1_{s(g)}= g$;
\item an inversion $inv : G_1 \to G_1$, $g \mto g^{-1}$ that produces the inverse arrow going from $s(g^{-1})=t(g)$ to $t(g^{-1})=s(g)$, fulfilling $g^{-1}g=1_{s(g)}$, $gg^{-1}=1_{t(g)}$.
\end{enumerate}
\end{defn}

These maps can be assembled into a diagram
\begin{equation}\label{eq 4.2.13}
\begin{tikzcd}
G_2 \arrow[r, "m"] 
& G_1 \arrow[r, "inv"] 
& G_1 \arrow[r, shift left,"s"]
\arrow[r, shift right,"t"']
& G_0 \arrow[r, "1"]
& G_1
\end{tikzcd}
\end{equation} 

An arrow may be denoted by $x \xrightarrow{g} y$ to indicate that $y=s(g)$ and $x=t(g)$.

\medskip\noindent

A topological groupoid is a groupoid in which both $G_1$ and $G_0$ are topological spaces and all the structure maps 
are continuous. Similarly one defines smooth groupoids, where in addition $s$ and $t$ are required to be 
surjective submersions in order to ensure that $G_2=G_1 ~^s\times^t_{G_0}G_1$ remains a manifold. A 
topological (or smooth) groupoid is called \'etale if the source map is a local homeomorphism (or local 
diffeomorphism); this condition implies that all structures maps are local homeomorphisms (or local diffeomorphisms, 
respectively). In the smooth case, this equivalently amounts to saying that $\dim G_1=\dim G_0$. In particular, 
an \'etale groupoid has zero-dimensional source and target fibers, and hence they are discrete. We shall only be 
dealing with smooth \'etale groupoids.

\medskip\noindent

We give some examples of \'etale groupoids below.

\begin{exa}\label{exa 4.2.11}\mbox{}
\begin{enumerate}[i)]
\item The unit groupoid has a single manifold $M$ as both its object and arrow space. All the maps are identity
functions.
\item A (discrete) group is a one-object groupoid (called the point groupoid).
\item The translation groupoid $\G \ltimes M$ of a smooth left action of a discrete group has as object space $M$ and arrow space $\G \times M$. The source is $(g,m) \mto m$, the target is $(g,m) \mto gm$ and the multiplication is $(g,m)(g',m')=(gg',m')$.
\item Orbifold groupoids or proper \'etale groupoids. We refer to \cites{MR2012261,MR3448330} for more details.
\item Let $(M,\cF)$ be a foliated manifold. Then the (reduced) holonomy groupoid is \'etale. 
\end{enumerate}
\end{exa}

As the last example is one of our main motivating examples, we shall describe it in a slightly greater details. See \cites{MR1151583,MR2012261,MR824240,MR1813430}. A foliation $\cF$ on $M$ is given by a cocycle $\cU=\{U_i,f_i,g_{ij}\}$ modeled on a manifold $N_0$ ($\bR^n$ or $\bC^n$), i.e.,
\begin{enumerate}[i)]
\item $\{U_i\}$ is an open covering of $M$;
\item $f_i : U_i \to N_0$ are submersions with connected fibers defining $\cF$;
\item $g_{ij}$ are local diffeomorphisms of $N_0$ and $g_{ij}f_j=f_i$ on $U_i \cap U_j$.
\end{enumerate}

The manifold $N=\sqcup f_i(U_i)$ is called the transverse manifold of $\cF$ associated to
the cocycle $\cU$, and the pseudogroup $P$ generated by $g_{ij}$ is called the holonomy pseudogroup on the transverse manifold. To any pseudogroup $P$ on some manifold $X$ we can associate an \'etale (effective) groupoid $\G(P)$ over $X$ as follows: for any $x,y \in X$ let 
\begin{equation}\label{eq 4.2.14}
\G(P)(x,y)=\{germ_xg \mid g \in P, x \in dom(g), g(x)=y\}.
\end{equation} The multiplication in $\G(P)$ is given by the composition of transitions. Equipped
with classical sheaf topology $\G(P)_1$ becomes a smooth manifold and $\G(P)$ becomes an \'etale groupoid. In our case, $\G(P)$ is called the reduced holonomy groupoid of $(M,\cF)$ and is denoted $Hol_N(M,\cF)$ (but
also we write $\G(P)$ sometimes).

\medskip\noindent

We now show one gets Hopf algebroids naturally from \'etale groupoids following \cites{MR2817646,MR2311185}. Before that we introduce the following.

\subsubsection*{\bf Fiber sum notation. } Let $E$ and $F$ are vector bundles over two manifolds $X$
and $Y$, respectively. Suppose $\p : X \to Y$ is an \'etale map (i.e., a local
homeomorphism) and $\a : E \isom \p^*F$ an isomorphism of vector bundles. Then the push-forward (or fiber sum) of $\p$, denoted by $\p_* : \G_c(X,E) \to \G_c(Y,F)$, is defined by 
\begin{equation}\label{eq 4.2.15}
(\p_*s)(y)=\sum_{\p(x)=y}\a(s(x)),
\end{equation} where $x \in X, y \in Y$ and $s \in \G_c(X,E)$. Here we identify the
fiber ${\p}^*F_z$ with $F_{\p(z)}$ using the definition of pullback.

\medskip\noindent

If $G$ is an \'etale groupoid over a compact Hausdorff $G_0$, the space $C_c^{\infty}(G)$ of smooth functions on $G=G_1$ with compact support carries a Hopf algebroid structure. Although $G=G_1$ often happens to be non-Hausdorff in examples, we assume this condition in this paper since the reduced holonomy groupoid of a Riemannian foliation is always Hausdorff. We have two $C^{\infty}(G_0)$-actions on $C_c^{\infty}(G)$ by left and right multiplication with respect to which we define the four tensor products denoted by $\tens^{ll}_{C^{\infty}(G_0)}$, $\tens^{rr}_{C^{\infty}(G_0)}$, $\tens^{rl}_{C^{\infty}(G_0)}$ and $\tens^{lr}_{C^{\infty}(G_0)}$. We need the following isomorphisms
\begin{equation}\label{eq 4.2.16} 
\begin{split}
\O_{s,t} : C_c^{\infty}(G) \tens^{rl}_{C^{\infty}(G_0)}C_c^{\infty}(G) \to C_c^{\infty}(G ~^s\times^t_{G_0}G)=C_c^{\infty}(G_2)\\
\O_{t,t} : C_c^{\infty}(G) \tens^{ll}_{C^{\infty}(G_0)}C_c^{\infty}(G)\to C_c^{\infty}(G ~^t\times^t_{G_0}G)=C_c^{\infty}(G_2)\\
\O_{s,s} : C_c^{\infty}(G) \tens^{rr}_{C^{\infty}(G_0)}C_c^{\infty}(G) \to C_c^{\infty}(G ~^s\times^s_{G_0}G)=C_c^{\infty}(G_2)\\
\O_{t,s} : C_c^{\infty}(G) \tens^{lr}_{C^{\infty}(G_0)}C_c^{\infty}(G) \to C_c^{\infty}(G ~^t\times^s_{G_0}G)=C_c^{\infty}(G_2)\\
\end{split}
\end{equation}
all given by the formulas
\begin{equation}\label{eq 4.2.17}
\O_{-.-}(u \tens^{--}_{C^{\infty}(G_0)} u')(g,g')=u(g)u(g'),
\end{equation} for $u,u' \in C_c^{\infty}(G)$ and $(g,g')$ in the respective pullback $G~^-\times^-_{G_0}G$. The maps are isomorphism, as it was shown in \cite{MR2311185}. We now give the Hopf algebroid structure maps
for $C_c^{\infty}(G)$ over $C^{\infty}(G_0)$:

\subsubsection*{Ring structure} On the base algebra $C^{\infty}(G_0)$ one has the commutative pointwise product, whereas the total algebra $C_c^{\infty}(G)$ is equipped with a convolution product, defined as the composition
\begin{equation}\label{eq 4.2.18}
\* : C_c^{\infty}(G) \tens^{rl}_{C^{\infty}(G_0)}C_c^{\infty}(G) \xrightarrow{\O_{s,t}} C_c^{\infty}(G_2) \xrightarrow{m_{\*}} C_c^{\infty}(G).
\end{equation} Explicitly,
\begin{equation}\label{eq 4.2.19}
(u \* v)(g):=\*(u \tens v)=(m_{\*}\O^{s,t}(u \tens v))(g)=\sum_{g=g_1g_2}u(g_1)u(g_2),
\end{equation} which can be used in showing associativity of the product.

\subsubsection*{Source and target maps} For $f \in C^{\infty}(G_0)$ and $u \in C_c^{\infty}(G)$,
\begin{equation}\label{eq 4.2.20}
(f \* u)(g)=f(t(g))u(g) \qquad \text{and} \qquad (u \* f)(g)=u(g)f(s(g)).
\end{equation} It can be shown that $C^{\infty}(G_0)$, identified with those functions in $C_c^{\infty}(G)$ having support on $1_{G_0} \subset G$, is a commutative subalgebra of $C_c^{\infty}(G)$. We put for the (left and right bialgebroid) source and target maps 
\begin{equation}\label{eq 4.2.21}
s_l\equiv s_r\equiv t_l\equiv t_r\equiv 1_{\*} : C^{\infty}(G_0) \to C_c^{\infty}(G),
\end{equation} i.e., the injection as subalgebra given by the fiber sum of the unit map $1 : G_0 \to G$.
More explicitly,
\begin{equation}\label{eq 4.2.22}
s_l : f \mto \ol{f}, \qquad \text{where} \qquad \ol{f}(g)= \begin{cases} f(x) & \text{if } g=1_x \text{ for some } x \in G_0\\ 0 & \text{otherwise.} \end{cases}
\end{equation}

\subsubsection*{Left and right comultiplications} Using the isomorphism $\O_{-,-}$, the left and right comultiplications are given as follows:
\begin{subequations}
\begin{equation}\label{eq 4.2.23a}
\begin{split}
\D_l : C_c^{\infty}(G) \to C_c^{\infty}(G~^t\times^t_{G_0}G)\isom C_c^{\infty}(G) \tens^{ll} C_c^{\infty}(G),\\
(\D_lu)(g,g')=\begin{cases} u(g) & \text{if } g=g',\\ 0 & \text{else},\end{cases}
\end{split}
\end{equation}
\begin{equation}\label{eq 4.2.23b}
\begin{split}
\D_r : C_c^{\infty}(G) \to C_c^{\infty}(G~^s\times^s_{G_0}G)\isom C_c^{\infty}(G) \tens^{rr} C_c^{\infty}(G),\\
(\D_ru)(g,g')=\begin{cases} u(g) & \text{if } g=g',\\ 0 & \text{else.}\end{cases}
\end{split}
\end{equation}
\end{subequations} Alternatively, $\D_l=d^l_{\*}$ and $\D_r=d^r_{\*}$ for the diagonal maps $d^l : G \to G~^t\times^t_{G_0}G$, $g \mto (g,g)$ and $d^r :G \to G~^s\times^s_{G_0}G$, $g \mto (g,g)$.

\subsubsection*{Left and right counits} Both left and right counits are respectively determined
by the fiber sum of the target and source maps of the groupoid. For any $x \in G_0$,
\begin{equation}\label{eq 4.2.24}
\begin{matrix}
\ve_l : C_c^{\infty}(G) \to C^{\infty}(G_0), & (\ve_l(u))(x)=\sum_{t(g)=x}u(g)\\
\ve_r : C_c^{\infty}(G) \to C^{\infty}(G_0), & (\ve_r(u))(x)=\sum_{s(g)=x}u(g).
\end{matrix}
\end{equation}

\subsubsection*{Antipode} The antipode is given by the groupoid inversion,
\begin{equation}\label{eq 4.2.25}
S : C_c^{\infty}(G) \to C_c^{\infty}(G), \qquad (S(u))(g)=u(g^{-1})=(inv_{\*}u)(g).
\end{equation}

\begin{thm}\label{thm 4.2.12}
With the above structure maps, $C_c^{\infty}(G)$ becomes a Hopf algebroid over $C^{\infty}(G_0)$.
\end{thm}

The proof is in \cite{MR2817646}. See also \cites{MR679730,MR823176,MR2456103,MR2530918}.

\subsection{More examples of Hopf algebroids} \label{sec 4.5}

So far we have focused on a single example, that of coming from \'etale groupoids. We have also mentioned Hopf algebras and weak Hopf algebras and built our framework using these as guiding examples. In this section we describe another example, namely, the Connes-Moscovici Hopf algebroid, which is over a noncommutative base, thus providing wider scope of our framework. Before we plunge into the Connes-Moscovici Hopf algebroid, we describe a special case, namely the following.

\subsubsection{The enveloping Hopf algebroid of an algebra}\label{subsec 4.5.1} Given an arbitrary $\bC$-algebra $A$, let $H=A \tens_{\bC} A^{op}$. The left bialgebroid structure over $A$ is given as 
\begin{subequations}
\begin{equation}\label{eq 4.5.1a}
s_l(a)=a \tens_{\bC} 1, \qquad t_l(b)=1 \tens_{\bC} b;
\end{equation}
\begin{equation}\label{eq 4.5.1b}
\D_l(a \tens b)=(a \tens_{\bC} 1) \tens_{A} (1 \tens_{\bC} b), \qquad \ve_l(a \tens_{\bC} b)=ab;
\end{equation}
\end{subequations} and the right bialgebroid structure over $A^{op}$ is given as
\begin{subequations}
\begin{equation}\label{eq 4.5.2a}
s_r(b)=1 \tens_{\bC} b, \qquad t_r(a)=a \tens_{\bC} 1;
\end{equation}
\begin{equation}\label{eq 4.5.2b}
\D_r(a \tens_{\bC} b)=(a \tens_{\bC} 1) \tens_{A^{op}} (1 \tens_{\bC} b), \qquad \ve_r(a \tens_{\bC} b)=ba;
\end{equation}
\end{subequations} for $a,b \in A$. Finally, the antipode
\begin{equation}\label{eq 4.5.3}
S(a \tens_{\bC} b)=b \tens_{\bC} a
\end{equation} makes $H$ into a Hopf algebroid. We now come to 

\subsubsection{The Connes-Moscovici Hopf algebroid}\label{subsec 4.5.2} Let $Q$ be a Hopf algebra over $\bC$ with antipode $T$ satisfying $T^2=\id$ and $A$ a $Q$-module algebra. Consider $H=A \tens_{\bC} Q \tens_{\bC} A$ with multiplication given by
\begin{equation}\label{eq 4.5.4}
(a \tens_{\bC} q \tens_{\bC} b)(a' \tens_{\bC} q' \tens_{\bC} b')=a(q_{(1)}a') \tens_{\bC} q_{(2)}q' \tens_{\bC} (q_{(3)}b')b.
\end{equation} for $a,b,a',b' \in A$ and $q,q' \in Q$. A left bialgebroid structure over $A$, known as the Connes-Moscovici bialgebroid, is given as 
\begin{subequations}
\begin{equation}\label{eq 4.5.5a}
s_l(a)=a \tens_{\bC} 1 \tens_{\bC} 1, \qquad t_l(b)=1 \tens_{\bC} 1 \tens_{\bC} b;
\end{equation}
\begin{equation}\label{eq 4.5.5b}
\D_l(a \tens_{\bC} q \tens_{\bC} b)=(a \tens_{\bC} q_{(1)} \tens_{\bC} 1) \tens_{A} (1 \tens_{\bC} q_{(2)} \tens_{\bC} b);
\end{equation}
\begin{equation}\label{eq 4.5.5c}
\ve_l(a \tens_{\bC} q \tens_{\bC} b)=a\ve(q)b;
\end{equation}
\end{subequations} for $a,b \in A$ and $q \in Q$. $\ve$ is the counit of $Q$ and we have used Sweedler notation for the coproduct of $Q$. This much is in the literature, see for example \cite{MR2553659}. We now put a right bialgebroid structure on $H$ over $A^{op}$ as
\begin{subequations}
\begin{equation}\label{eq 4.5.6a}
s_r(b)=1 \tens_{\bC} 1 \tens_{\bC} b, \qquad t_r(a)=a \tens_{\bC} 1 \tens_{\bC} 1;
\end{equation}
\begin{equation}\label{eq 4.5.6b}
\D_r(a \tens_{\bC} q \tens_{\bC} b)=(a \tens_{\bC} q_{(1)} \tens_{\bC} 1) \tens_{A^{op}} (1 \tens_{\bC} q_{(2)} \tens_{\bC} b);
\end{equation}
\begin{equation}\label{eq 4.5.6c}
\ve_r(a \tens_{\bC} q \tens_{\bC} b)=T(q)(ba);
\end{equation}
\end{subequations} for $a,b \in A$ and $q \in Q$. Observe that, for the above structure maps, \[a_1 \cdot (b \tens_{\bC} q \tens_{\bC} b')\cdot a_2=b(q_{(1)} \cdot a_1) \tens_{\bC} q_{(2)} \tens_{\bC} (q_{(3)} \cdot a_2)b',\] for $a_1,a_2,b,b' \in A$ and $q \in Q$. From this, it at once follows that $\D_r$ and $\ve_r$ are bimodule morphisms. Coassociativity of $\D_r$ and counitarity of $\ve_r$ are easy to verify. We now check the Takeuchi condition. Given $a,b,c \in A$ and $q \in Q$, we have
\[
\begin{aligned}
&s_r(a)(b \tens_{\bC} q_{(1)} \tens_{\bC} 1) \tens_{A^{op}} (1 \tens_{\bC} q_{(2)} \tens_{\bC} c)\\
=&((1 \tens_{\bC} 1 \tens_{\bC} a)(b \tens_{\bC} q_{(1)} \tens_{\bC} 1)) \tens_{A^{op}} (1 \tens_{\bC} q_{(2)} \tens_{\bC} c)\\
=&(b \tens_{\bC} q_{(1)} \tens_{\bC} a) \tens_{A^{op}} (1 \tens_{\bC} q_{(2)} \tens_{\bC} c)\\
=&((b \tens_{\bC} q_{(1)} \tens_{\bC} 1)(1 \tens_{\bC} 1 \tens_{\bC} T(q_{(2)})a)) \tens_{A^{op}} (1 \tens_{\bC} q_{(3)} \tens_{\bC} c)\\
=&(b \tens_{\bC} q_{(1)} \tens_{\bC} 1) \tens_{A^{op}} (1 \tens_{\bC} q_{(3)} \tens_{\bC} c)(T(q_{(2)})a \tens_{\bC} 1 \tens_{\bC} 1)\\
=&(b \tens_{\bC} q_{(1)} \tens_{\bC} 1) \tens_{A^{op}} (q_{(3)}T(q_{(2)})a \tens_{\bC} q_{(4)} \tens_{\bC} c)\\
=&(b \tens_{\bC} q_{(1)} \tens_{\bC} 1) \tens_{A^{op}} (a \tens_{\bC} q_{(2)} \tens_{\bC} c) \quad (\text{we use that } T^2=\id)\\
=&(b \tens_{\bC} q_{(1)} \tens_{\bC} 1) \tens_{A^{op}} ((a \tens_{\bC} 1 \tens_{\bC} 1)(1 \tens_{\bC} q_{(2)} \tens_{\bC} c))\\
=&(b \tens_{\bC} q_{(1)} \tens_{\bC} 1) \tens_{A^{op}} t_r(a)(1 \tens_{\bC} q_{(2)} \tens_{\bC} c), 
\end{aligned}
\] thus proving the Takeuchi condition. The verification of the character property of $\ve_r$ is left to the reader. So this proves that we indeed have a right bialgebroid. Now we define the antipode $S$ as
\begin{equation}\label{eq 4.5.7}
S(a \tens_{\bC} q \tens_{\bC} b)=T(q_{(3)})b \tens_{\bC} T(q_{(2)}) \tens_{\bC} T(q_{(1)})a.
\end{equation} Again, the antipode axioms are straightforward to check. As an example we show that $\m(S \tens \id_H)\D_l=s_r\ve_r$ holds:
\[
\begin{aligned}
&\m(S \tens \id_H)\D_l(a \tens_{\bC} q \tens_{\bC} b)\\
=&S(a \tens_{\bC} q_{(1)} \tens_{\bC} 1)(1 \tens_{\bC} q_{(2)} \tens_{\bC} b)\\
=&(T(q_{(3)})1 \tens_{\bC} T(q_{(2)}) \tens_{\bC} T(q_{(1)})a)(1 \tens_{\bC} q_{(4)} \tens_{\bC} b)\\
=&(1 \tens_{\bC} T(q_{(2)}) \tens_{\bC} T(q_{(1)})a)(1 \tens_{\bC} q_{(3)} \tens_{\bC} b)\\
=&T(q_{(4)})1 \tens_{\bC} T(q_{(3)})q_{(5)} \tens_{\bC} T(q_{(2)})bT(q_{(1)})a\\
=&1 \tens_{\bC} T(q_{(2)})q_{(3)} \tens_{\bC} T(q_{(1)})(ba)\\
=&1 \tens_{\bC} 1 \tens_{\bC} T(q)(ba)\\
=&s_r\ve_r(a \tens_{\bC} q \tens_{\bC} b).
\end{aligned}
\] 
Thus we have:

\begin{thm}\label{thm 4.5.1}
With the structures described above, $H$ becomes a Hopf algebroid, which we call the Connes-Moscovici Hopf algebroid.
\end{thm}

\begin{rem}\label{rem 4.5.2}
Observe that taking $Q=\bC$ gives the enveloping Hopf algebroid back and $A=\bC$ reduces $H$ to a Hopf algebra. Thus it is a simultaneous generalization of the cases discussed above.
\end{rem}

\begin{rem}\label{rem 4.5.3}
We have used $T^2=\id$ to make $H$ into a Hopf algebroid. We think that it is possible to remove this condition by introducing a ``modular pair in involution", that in turn produces a ``twisted antipode" for $Q$, hence for $H$. 
\end{rem}


\section{Hopf algebroids as symmetries} \label{sec 3}
In this section, we recall Hopf algebroids representation and introduce $\*$-structures, necessary for viewing Hopf algebroids as symmetry objects.

\subsection{Modules over Hopf algebroids}

Let $H=(H_l,H_r,S)$ be a Hopf algebroid. A left module over $H$ is simply a left module over the underlying $\bC$-algebra $H$. We denote the structure map by $(h,m) \mto h \cdot m$. The left bialgebroid structure $H_l$ induces an $(A_l,A_l)$-bimodule structure on each module and a monoidal structure on the category of modules. More explicitly, let $M$ be an $H$-module. Then the $(A_l,A_l)$-bimodule structure is given by 
\begin{equation}\label{eq 4.2.26}
a_1 \cdot m \cdot a_2=s_l(a_1) \cdot t_l(a_2) \cdot m,
\end{equation}
for all $a_1,a_2 \in A_l$ and $m \in M$. The left coproduct defines the monoidal structure $(M,N) \mto M \tens_A N$, where $M \tens_A N$ is equipped with the $H$-module structure 
\begin{equation}\label{eq 4.2.27}
h \cdot (m \tens n):=h_{(1)} \cdot m \tens h_{(2)} \cdot n, \qquad h \in H, m \in M, n \in N.
\end{equation}

The monoidal unit is given by $A_l$ with left $H$-action $h \cdot a=\ve_l(hs_l(a)).$ Note that $\ve_l(ht_l(a))=\ve_l(hs_l(\ve_l(t_l(a))))=\ve_l(hs_l(a)).$ Also $A_l$ being the monoidal unit it is an algebra in the category of $H$-modules, i.e., it is an $H$-module algebra. This structure will be important for us in the examples we consider.

\begin{rem}\label{rem 4.2.13}
We state the definition of an $H$-module algebra explicitly. It is a $\bC$-algebra and left $H$-module $B$ such that the multiplication in $B$ is $A_l$-balanced and 
\begin{enumerate}[i)]
\item $h \cdot 1_B=s_l\ve_l(h) \cdot 1_B$;
\item $h \cdot (bb')=(h_{(1)} \cdot b)(h_{(2)} \cdot b')$.
\end{enumerate} for $b,b' \in B$ and $h \in H$. Note that $B$ has a canonical $A_l$-ring structure. Its unit is the map $A_l \to B$, $a \mto s_l(a) \cdot 1_B=t_l(a) \cdot 1_B$.
\end{rem}

Similarly, one can consider right $H$-modules as modules over the $\bC$-algebra $H$. Such modules get the structure of an $(A_r,A_r)$-bimodule and the category becomes monoidal using the right coproduct. The monoidal unit is $A_r$. We now see some examples coming from the geometry of groupoids. We follow \cite{MR2853081}.

\begin{defn}\label{defn 4.2.14}
A smooth left action of a Lie groupoid $G$ on a smooth manifold $P$ along a smooth map $\pi : P \to G_0$ is a smooth map $\m : G_1 {}^s\times^{\pi}_{G_0}P \to P$, $(g,p) \mto g \cdot p$, which satisfies
the conditions $\pi(g \cdot p)=t(g)$, $1_{\pi(p)} \cdot p=p$ and $g' \cdot (g \cdot p)=(g'g) \cdot p$ for all $g',g \in G_1$ and $p \in P$ with $s(g')=t(g)$ and $s(g)=\pi(p)$.
\end{defn} 

We define right actions of \'etale groupoids on smooth manifolds in a similar way.

\begin{defn}\label{defn 4.2.15}
Let $G$ be an \'etale groupoid, and let $E$ be a smooth complex vector bundle over $G_0$. A representation of the groupoid $G$ on $E$ is a smooth left action $\r : G_1~^s\times^p_{G_0} E \to E$, denoted by
$\r(g,v)=g \cdot v$, of $G$ on $E$ along the bundle projection $p : E \to G_0$ such that for any arrow
$x \xrightarrow{g} y$ the induced map $g_* : E_x \to E_y$, $v \mto g \cdot v$, is a linear
isomorphism. A section $u : G_0 \to E$ is called $G$-invariant if for any arrow $x \xrightarrow{g} y$, it holds that $g \cdot u(x)=u(y)$. 
\end{defn}

Let us see what representations mean in the examples above.

\begin{exa}\label{exa 4.2.16}\mbox{}
\begin{enumerate}[i)]
\item Representations of the unit groupoid associated to a smooth manifold correspond precisely to complex vector bundles.
\item Representations of the point groupoid associated to a (discrete) group $\G$ correspond to representations of the group on finite dimensional complex vector spaces.
\item Representations of the translation groupoid $\G \ltimes M$ corresponds to $\G$-equivariant complex vector bundles over $M$.
\item Representations of the orbifold groupoid are the orbibundles.
\item Representations of the holonomy groupoid are the transversal vector bundles.
\item For an \'etale groupoid $G$ the complexified tangent bundle of $G_0$ becomes a representation of $G$. The cotangent bundle, exterior bundle all inherit this natural representation, so it makes sense to speak of vector fields, differential forms or Riemannian metrics etc. on \'etale groupoids (vector fields, differential forms or Riemannian metrics etc. on $G_0$, respectively, invariant under the action). Also note that the exterior derivative $d$ is invariant under the $G$-action. This follows from naturality of $d$ and a local argument.
\end{enumerate}
\end{exa}

\begin{prop}\label{prop 4.2.17}
Let $E$ be representation of the \'etale groupoid $G$. The space of smooth sections $\G^{\infty}(E)$
over $G_0$ becomes a module over $C_c^{\infty}(G)$ by the formulas
\begin{equation}\label{eq 4.2.28}
(a \cdot u)(x)=\sum_{t(g)=x}a(g)(g \cdot u(s(g))),
\end{equation}
for $a \in C_c^{\infty}(G)$ and $u \in \G^{\infty}(E)$.
\end{prop}

The proof is in \cite{MR2853081}. Moreover, each module of finite type and constant rank appears in this way, giving a version of Serre-Swan theorem. See \cite{MR823176} for an example coming from Sobolev spaces.

\subsection{\texorpdfstring{$\*$}{}-structures and conjugate modules}

We introduce $\*$-structures on Hopf algebroids which will be needed in order to view them as symmetry objects. This is one of the main results of the present paper. We view the ensuing structures as the first step in defining a ``compact"-type Hopf algebroid in analogy with $CQG$-algebras \cite{MR1310296}, though we do not go in that direction here.

\medskip\noindent

Let $(H_l,H_r,S)$ be a Hopf algebroid such that $H$, $A_l$ and $A_r$ are $\*$-algebras, $s_l$ and $s_r$ are $\*$-preserving (the involutions for $H$, $A_r$ and $A_l$ are denoted by the same symbol $\*$). Assume that
\begin{equation}\label{eq 4.2.29}
\ve_lt_r(a_1^*)=(\ve_ls_r(a_1))^*, \qquad \ve_rt_l(a_2^*)=(\ve_rs_l(a_2))^*
\end{equation}
hold for all $a_1 \in A_r$, $a_2 \in A_l.$

\begin{lem}\label{lem 4.2. 18}
We have
\begin{equation}\label{eq 4.2.30}
h^*t_l(a)^* \tens_{A_r} h'^*=h^* \tens_{A_r} h'^*s_l(a)^*.
\end{equation}
\end{lem}

\begin{proof}
We compute
\begin{equation}\label{eq 4.2.31}
\begin{aligned}
h^*t_l(a)^* \tens_{A_r} h'^*&=h^*s_r(\ve_r(t_l(a)))^* \tens_{A_r} h'^*\\
&=h^*s_r((\ve_rt_l(a))^*) \tens_{A_r} h'^*\\
&=h^*s_r\ve_rs_l(a^*) \tens_{A_r} h'^*\\
&=h^* \cdot \ve_rs_l(a^*) \tens_{A_r} h'^*\\
&=h^* \tens_{A_r} \ve_rs_l(a^*) \cdot h'^*\\
&=h^* \tens_{A_r} h'^*t_r\ve_rs_l(a^*)\\
&=h^* \tens_{A_r} h'^*s_l(a^*)\\
&=h^* \tens_{A_r} h'^*s_l(a)^*.
\end{aligned}
\end{equation}
\end{proof}

Lemma \ref{lem 4.2. 18} says that the map $(\* \tens \*) : H_l \tens_{\bC} H_l \to H_r \tens_{A_r} H_r$ descends to an isomorphism $(\* \tens \*) : H_l \tens_{A_l} H_l \to H_r \tens_{A_r} H_r$. So we can make sense of
\begin{equation}\label{eq 4.2.32}
\D_r \*=(\* \tens \*)\D_l.
\end{equation}

In Sweedler notation,
\begin{equation}\label{eq 4.2.33}
(h^*)^{(1)} \tens (h^*)^{(2)}=(h_{(1)})^* \tens (h_{(2)})^*\, .
\end{equation}

\begin{defn}\label{defn 4.2.19}
Let $(H_l,H_r,S)$ be a Hopf algebroid such that $H$, $A_l$ and $A_r$ are $\*$-algebras while $s_l$ and $s_r$ are $\*$-preserving. Then $(H_l,H_r,S)$ is said to be a Hopf $\*$-algebroid if \eqref{eq 4.2.29} and \eqref{eq 4.2.32} hold.
\end{defn}

Some immediate corollaries of Definition \ref{defn 4.2.19} are:
\begin{enumerate}[i)]
\item $(\* \tens \*) : H_l \tens_{A_l} H_l \to H_r \tens_{A_r} H_r$ induces an isomorphism $H_l \times _{A_l} H_l \to H_r \times^{A_r} H_r$.
\item From \eqref{eq 4.2.5}, $t_l\*=s_r\ve_rt_l\*=s_r\*\ve_rs_l=\*s_r\ve_rs_l=\*Ss_l$, with
the last equality following from \eqref{eq 4.2.11}.
\item Similarly, $t_r\*=\*Ss_r.$
\end{enumerate}

\begin{prop}\label{prop 4.2.20}
Let $(H_l,H_r,S)$ be a Hopf $\*$-algebroid. Then the counits and the antipode satisfy
\begin{equation}\label{eq 4.2.34}
\ve_rS^{-1}\*=\*\ve_r, \qquad \ve_lS^{-1}\*=\*\ve_l, \qquad S\*S\*=\id_H
\end{equation}
and $A_l$ becomes an $H$-module $\*$-algebra, i.e., the $H$-action satisfies
\begin{equation}\label{eq 4.2.35}
(h \cdot a)^*=S(h)^* \cdot a^* \qquad h \in H, \quad a \in A_l.
\end{equation}
\end{prop}

\begin{proof}
We have \[h^*=s_l\ve_l((h^*)_{(1)})(h^*)_{(2)}=s_l\ve_l((h^{(1)})^*)(h^{(2)})^*,\] so \[h=h^{(2)}(s_l\ve_l((h^{(1)})^*))^*.\] Similarly, \[h=h^{(1)}(t_l\ve_l((h^{(2)})^*))^*.\] Now, \[\*s_l\ve_l\*=s_l\*\ve_l\*=t_r\ve_rs_l\*e_l\*=t_r\*\ve_rt_l\ve_l\*.\]
Similarly, 
\[\*t_l\ve_l\*=s_r\*\ve_rt_l\ve_l\*.\]
So we conclude that $\*\ve_rt_l\ve_l\*$ satisfies the right counit axioms. Hence $\ve_r=\*\ve_rt_l\ve_l\*=\*\ve_rS^{-1}\*$. Similarly, $\ve_l=\*\ve_lS^{-1}\*.$ From this we observe
that $s_l\ve_l\*=\*t_r\ve_r$ and $s_r\ve_r\*=\*t_l\ve_l$. Using the above observation and proceeding exactly as before, it follows that $\*S^{-1}\*$ satisfies the antipode axioms. By uniqueness,
we have $S=\*S^{-1}\*$, which implies $S\*S\*=\id_H$.

\medskip\noindent
Finally,
\[ 
\begin{aligned}
S(h)^* \cdot a^*&=\ve_l(S(h)^*s_l(a^*))\\
&=\ve_l(S(h)^*s_l(a)^*)\\
&=\ve_l\*(s_l(a)S(h))\\
&=\ve_l\*(St_l(a)S(h))\\
&=\ve_l\*S(ht_l(a))\\
&=\*\ve_l(ht_l(a))\\
&=(\ve_l(hs_l(a)))^*\\
&=(h \cdot a)^*.
\end{aligned}
\]
\end{proof}

Besides Hopf $\*$-algebras, the Hopf algebroid in Theorem \ref{thm 4.2.12} becomes a central example of Hopf $\*$-algebroids:

\begin{prop}\label{prop 4.2.21}
The space $C_c^{\infty}(G)$ becomes a Hopf $\*$-algebroid over $C^{\infty}(G_0)$ with $\*$-structure given by
\begin{equation}\label{eq 4.2.36}
u^*(g)=\ol{u(g^{-1})} \text{ for } u \in C_c^{\infty}(G) \text{ and } f^*(x)=\ol{f(x)} \text{ for } f \in C^{\infty}(G_0).
\end{equation}
\end{prop}

\begin{proof}
This follows from direct computations.
\end{proof}

If $A$ is a $\*$-algebra then $A \tens_{\bC}A^{op}$ as in Subsubsection \ref{subsec 4.5.1} is a Hopf $\*$-algebroid. Furthermore, if $Q$ is a Hopf $\*$-algebra and $A$ is a $Q$-module $\*$-algebra then the Connes-Moscovici Hopf algebroid in Theorem \ref{thm 4.5.1} $H$ becomes a Hopf $\*$-algebroid in our sense. 
Another class of examples, which we have not mentioned above, comes from weak Hopf algebras studied in \cite{MR1726707}. Our $\*$-structure is the same as $C^*$-structure mentioned in \cite{MR1726707}. Following this and the standard theory of $CQG$-algebras, leads to opening up a new direction of study, namely, (co)representation theory of Hopf $\*$-algebroids and the interplay of the $\*$-structure and (co)integrals.

\medskip\noindent

We shall systematically use the language of conjugate modules in order to keep track of various
aspects, see \cites{MR2501177,MR3073899}. Let $(H_l,H_r,S)$ be a Hopf $\*$-algebroid and $M$ an $H$-module. We define the conjugate module $\ol{M}$ by declaring that
\begin{enumerate}[i)]
\item $\ol{M}=M$ as abelian group;
\item we write $\ol{m}$ for an element $m \in M$ when we consider it as an element of $\ol{M}$;
\item the module operation for $\ol{M}$ is $h \cdot \ol{m}=\ol{S(h)^* \cdot m}$.
\end{enumerate}

Again, let $B$ be a $\*$-algebra and let $E$ be a $(B,B)$ bimodule. The conjugate bimodule $\ol{E}$ is defined
by the following three conditions:
\begin{enumerate}[i)]
\item $\overline{E}=E$ as abelian group;
\item We write $\ol{e}$ for an element $e \in E$ when we consider it as an element of $\ol{E}$;
\item The bimodule operations for $\ol{E}$ are $b \cdot \ol{e}=\overline{e \cdot b^*}$ and $\ol{e} \cdot b=\ol{b^* \cdot e}.$
\end{enumerate}

If $\ta : E \to F$ is any morphism, then we define $\ol{\ta} : \ol{E} \to \ol{F}$ by $\ol{\ta}(\ol{e})=\ol{\ta(e)}$. We make $\ol{B}$ an associative algebra by defining the multiplication $\ol{b}{}\ol{b'}:=\ol{b'b}$. As an $\bR$-algebra, $\ol{B}$ is isomorphic to $B^{op}$ via the map $b \mto \ol{b}$. We make $\ol{B}$ a $\bC$-algebra through the algebra homomorphism $\bC \to \ol{B}$, $\l \mto \ol{\l^*1_B}$. We now define $\# : B \to \ol{B}$, $b \mto \ol{b^*}$. Then $\#$ is an isomorphism of $\bC$ algebras. If $\ta : B \to B'$ is a morphism then we say that $\ta$ is $\*$-preserving if $\#\ta=\ol{\ta}\#$.

\medskip\noindent

So we see that the conclusion in \eqref{eq 4.2.35} that $A_l$ is an $H$-module $\*$ algebra, is nothing but the assertion that 
$\# : A_l \to \ol{A_l}$ is an $H$-module morphism. We also see that for an $H$-module $M$, the induced 
$(A_l,A_l)$-bimodule structure matches with the prescription above. Thus our $\*$-structure naturally produces 
examples of ``Bar categories" in the sense of \cite{MR2501177}.

\begin{lem}\label{lem 4.2.22}
Let $B$ be an $H$-module $\*$-algebra, and let the invariant subalgebra $B_H$ be defined as $B_H=\{b \in B \mid h \cdot b=s_l\ve_l(h) \cdot b\}$. Then $\# : B \to \ol{B}$ induces an isomorphism $\# : B_H \to \ol{B}_H$.
\end{lem}

\begin{proof}
This follows from the fact that $\#$ is an $H$-module morphism.
\end{proof}

In fact, we can say more:

\begin{prop}\label{prop 4.2.23}
Let $B$ be an $H$-module $\*$-algebra. Then $B_H$ is also a $\*$-algebra. So that, by Lemma \ref{lem 4.2.22} we can identify $(\ol{B})_H=\ol{(B_H)}$ as algebras.
\end{prop}

\begin{proof}
Let $b \in B_H$. We compute
\[
\begin{aligned}
(h \cdot b^*)^*&=S(h)^* \cdot b\\
&=s_l\ve_l(S(h)^*) \cdot b\\
&=s_l(\ve_l(h)^*) \cdot b\\
&=(s_l\ve_l(h))^* \cdot b
\end{aligned}
\] so that 
\[
\begin{aligned}
h \cdot b^*&=((s_l\ve_l(h))^* \cdot b)^*\\
&=(S(s_l\ve_l(h)^*))^* \cdot b^*\\
&=S^{-1}s_l\ve_l(h) \cdot b^*\\
&=t_l\ve_l(h) \cdot b^*.
\end{aligned}
\] Next observe that taking $h=s_l(a)$ for $a \in A_l$ in the last equality gives $s_l(a) \cdot b^*=t_l(a) \cdot b^*$. So that for all $h \in H$ we get \[s_l\ve_l(h) \cdot b^*=t_l\ve_l(h) \cdot b^*\] which in turn implies that $b^* \in B_H$.
\end{proof}

\section{Applications to noncommutative K\"ahler structures} \label{sec 4.3}
In this section, we introduce Hopf algebroid action on differential graded algebras, a not-so-straightforward generalization of the Hopf algebra case. Then we introduce noncommutative K\"ahler structures admitting Hopf algebroid equivariance, generalizing the case of foliations along the way.

\subsection{Noncommutative K\"ahler structures}
We begin with a brief review of noncommutative K\"ahler structures in the sense of \cite{MR3720811}.

\subsubsection{Complexes and Double Complexes}\label{subsec 3.1.2} For $(S,+)$ a commutative semigroup, an $S$-graded algebra is an algebra of the form $A=\bigoplus_{s \in S}A^s$, where each $A^s$ is a linear
subspace of $A$, and $A^sA^t \subset A^{s+t}$, for all $s,t
\in S$. If $a \in A^s$, then we say that $a$ is a homogeneous element of degree $s$.  A homogeneous mapping of degree $t$ on $A$ is a linear mapping $L : A \to A$ such that if $a \in
A^s$, then $L(a) \in A^{s+t}$. We say that a subspace $B$ of $A$ is  homogeneous if it admits a decomposition  $B=\oplus_{s \in S} B^s$, with $B^s \subset A^s$, for all $s \in S$.

\medskip\noindent

A pair $(A,d)$ is called a complex if $A$ is an $\bN_0$-graded algebra, and $d$ is a homogeneous mapping  of degree $1$ such that $d^2 = 0$. A triple $(A,\del,\adel)$ is called a double complex if $A$ is an $\bN^2_0$-graded algebra, $\del$ is homogeneous mapping of degree $(1,0)$, $\adel$ is homogeneous mapping of degree $(0,1)$, and
\begin{equation}\label{eq 3.1.1}
\del^2=\adel^2=0, \quad \del\adel=-\adel\del.
\end{equation} 





\subsubsection{Differential $*$-Calculi}\label{subsec 3.1.3}

A complex $(A,d)$ is called a differential graded algebra if $d$ is a graded derivation, which is to say, if it satisfies the graded Leibniz rule
\begin{equation}\label{eq 3.1.5}
d(\a\b)=d(\a)\b+(-1)^k \a d(\b),       
\end{equation} for all $\a \in A^k$, $\b \in A$. The operator $d$ is called the differential of the differential graded algebra. 

\begin{defn}\label{defn 3.1.2}
A differential calculus over an algebra $A$ is a differential graded
algebra $(\O,d)$ such that $\O^0=A$, and

\begin{equation}\label{eq 3.1.6}	
\O^k=span_{\bC}\{a_0da_1\wdg \dots \wdg da_k \mid a_0,\dots,a_k \in A\}.
\end{equation}
\end{defn}

We use $\wdg$ to denote the multiplication between elements of a differential calculus when both are of order greater than $0$. We call an element of a differential calculus a form. A differential map between two differential calculi $(\O,\d_{\O})$ and $(\G,d_{\G})$, defined over the same algebra $A$, is a bimodule map $\f: \O  \to \G$ such that $\f d_{\O}=d_{\G}\f$. 

\medskip\noindent

We call a differential calculus $(\O,d)$ over a $*$-algebra $A$ a differential $\*$-calculus if the involution of $A$ extends to an involutive conjugate-linear map on $\O$, for which $(d\w)^* = d\w^*$, for all $\w \in \O$, and
\begin{equation}\label{eq 3.1.7}
(\w \wdg \nu)^*=(-1)^{kl}\nu^* \wdg \w^*,
\end{equation} for all $\w \in
\O^k$, $\nu \in \O^l$. We say that a form $\w \in \O$ is real  if $\w^* = \w$.


\subsubsection{Orientability and Closed Integrals}\label{subsec 3.1.4} 

We say that a differential calculus has total dimension $n$ if $\O^k=0$, for all $k>n$, and $\O^n \neq 0$. If in addition there exists an $(A,A)$-bimodule isomorphism $\vol: \O^n \simeq A$,  then we say that $\O$ is orientable. We call a choice of such an isomorphism an orientation. 
If  $\O$ is a $\*$-calculus over a $\*$-algebra, then a $\*$-orientation is an orientation which is also a $*$-map. A $\*$-orientable calculus is one which  admits a $\*$-orientation.



\subsubsection{Complex structures}\label{subsec 3.2.1} We now recall the definition of a noncommutative complex structure.

\begin{defn}\label{defn 3.2.1}
An  almost complex structure for a differential $\*$-calculus  $\O$, over a $\*$-algebra $A$, is an $\bN^2_0$-algebra grading $\bigoplus_{(a,b)\in \bN^2_0} \O^{(a,b)}$ for $\O$ such that 
\begin{enumerate}[i)]
\item $\O^k=\bigoplus_{a+b=k}\O^{(a,b)}$, for all $k \in \bN_0$,
\item $(\O^{(a,b)})^*=\O^{(b,a)}$, for all $(a,b) \in \bN^2_0$.
\end{enumerate}
\end{defn}

We call an element of $\O^{(a,b)}$ an $(a,b)$-form. Let $\del$ and $\adel$ be the unique homogeneous operators  of order $(1,0)$, and $(0,1)$ respectively, defined by 
\begin{equation}\label{eq 3.2.1}
\del|_{\O^{(a,b)}}:=\proj_{\O^{(a+1,b)}}d, \quad \adel|_{\O^{(a,b)}}:=\proj_{\O^{(a,b+1)}}d,
\end{equation} where $\proj_{\O^{(a+1,b)}}$, and $\proj_{\O^{(a,b+1)}}$, are the projections from $\O^{a+b+1}$ onto $\O^{(a+1,b)}$, and $\O^{(a,b+1)}$, respectively. The proof of the following lemma carries over directly from the classical setting \cite{MR2093043}.

\begin{lem}\label{lem 3.2.2} 
If $\bigoplus_{(a,b)\in \bN^2_0}\O^{(a,b)}$ is an almost complex structure for a differential $\*$-calculus $\O$ over an algebra $A$, then the following two conditions are equivalent:
\begin{enumerate}[i)]
\item $d=\del+\adel$,
\item the triple $\big(\bigoplus_{(a,b)\in \bN^2}\O^{(a,b)}, \del,\adel\big)$ is a double complex.
\end{enumerate}
\end{lem}

With this in hand,

\begin{defn}\label{defn 3.2.3}
When the conditions in Lemma \ref{lem 3.2.2} hold for an almost complex structure, then we say that it is  integrable. 
\end{defn}

We call an integrable almost complex structure a  complex structure, and the double complex  $(\bigoplus_{(a,b)\in \bN^2}\O^{(a,b)},\del,\adel)$  its  Dolbeault double complex. An easy consequence of integrability is that 
\begin{equation}\label{eq 3.2.2}
\del(\w^*) = (\adel \w)^*, \quad  \adel(\w^*) = (\del \w)^*,
\end{equation} for all $\w \in \O$.

\subsubsection{Hermitian and K\"ahler structures}\label{subsec 3.2.2} 

Let $\O$ denote a differential $\*$-calculus, over an algebra $A$, of total dimension $2n$. As a first step towards the definition  of a hermitian form, we present a direct noncommutative generalization of the classical definition of an almost symplectic form.

\begin{defn}\label{defn 3.2.4}
An  almost symplectic form for $\O$ is a central real $2$-form $\s$ such that, with respect to the  Lefschetz operator
\begin{equation}\label{eq 3.2.3}
L :\O \to \O,  \quad   \w \mto \s \wdg \w,
\end{equation}
isomorphisms are given by
\begin{equation} \label{eq 3.2.4}
L^{n-k}: \O^{k} \to  \O^{2n-k}, 
\end{equation} for all $1 \leq k < n$.
\end{defn}

Note that since $\s$ is a central real form, $L$ is an $(A,A)$-bimodule $*$-homomorphism. 

\begin{defn}\label{defn 3.2.5}
For $L$ the Lefschetz operator of any almost symplectic form,  the space of  primitive $k$-forms  is
\begin{equation}\label{eq 3.2.5}
P^k:=\{\a \in \O^{k} \mid L^{n-k+1}(\a)=0\},  \text{ ~ if } k \leq n,\quad \text{ and } \quad P^k := 0, \text{ ~ if } k>n.
\end{equation}
\end{defn}



Now we present the following noncommutative generalization of the classical notion of a symplectic form \cite{MR2093043}.

\begin{defn}\label{defn 3.2.7}
A  symplectic form is a $d$-closed almost symplectic form.
\end{defn}

We now introduce a hermitian structure for a differential $\*$-calculus, which is  essentially just a symplectic form interacting with a complex structure in a natural way.  In the commutative case each such form is the fundamental form of a uniquely identified hermitian metric  \cite{MR2093043}.

\begin{defn}\label{defn 3.2.8}
An  hermitian structure for a $\*$-calculus $\O$  is a pair $(\O^{(\cdot,\cdot)}, \s)$ where $\O^{(\cdot,\cdot)}$ is a complex structure  and  $\s$ is an almost symplectic form, called the  hermitian form, such that $\s \in \O^{(1,1)}$.
\end{defn}

\begin{defn}
A  K\"ahler structure for a differential $*$-calculus is a hermitian structure $(\O^{(\cdot,\cdot)},\k)$ such that the hermitian form $\k$ is $d$-closed. We call such a $\k$ a   K\"ahler form.
\end{defn}
	
Every $2$-form in a $*$-calculus with total dimension $2$ is obviously $d$-closed. Hence, just as in the classical case, with respect to any choice of complex structure, every $\k \in \O^{(1,1)}$ is a K\"ahler form. 

\subsection{Covariant noncommutative K\"ahler structures}
In this subsection, we introduce Hopf algebroid covariance into the setup of the previous subsection.

\subsubsection{Covariant differential calculi}
\label{subsec 4.3.1} Let $H=(H_l,H_r,S)$ be a Hopf $\*$-algebroid. We start by defining a differential calculus. 

\begin{defn}\label{defn 4.3.1}
An $\bN_0$-graded $H$-module is an $\bN_0$-graded $\bC$-vector space which is also an $H$-module such that the
$H$-action preserves the $\bN_0$-grading.
\end{defn}

\begin{defn}\label{defn 4.3.2}
An $\bN_0$-graded $H$-module algebra is an $\bN_0$-graded algebra
which is also an $H$-module algebra such that the $H$-action preserves the $\bN_0$-grading.
\end{defn}

\begin{defn}\label{defn 4.3.3}
A pair $(B,\,d)$ is called an $H$-covariant complex if $B$ is an $\bN_0$-graded $H$-module algebra, and $d$ is homogeneous of degree
one satisfying $d^2=0$, such that 
\begin{equation}\label{eq 4.3.1}
\text{$A_l$ and $H_0$ generate H as algebra,}
\end{equation}
where
\begin{equation}\label{eq 4.3.1n}
H_0:=\{h \in H \mid [h-s_l\ve_l(h),d]=[h-t_l\ve_l(h),d]=0\}\,.
\end{equation}
\end{defn}

\begin{defn}\label{defn 4.3.4} A triple $(B,\del,\adel)$ is called an $H$-covariant double complex if $B$ is an $\bN_0^2$-graded 
$H$-module algebra, $\del$ is homogeneous of degree $(1,0)$, and $\adel$ is homogeneous of degree $(0,1)$, such that $\del^2=0$, 
$\adel^2=0$, $\del \adel +\adel \del=0$ and they satisfy \eqref{eq 4.3.1}. \end{defn}

For any $H$-covariant complex $(B,\,d)$, we call an element $d$-closed if it is contained in $\ker(d)$ and $d$-exact if it is contained in $im(d)$. For an $H$-covariant double complex $(B,\del,\adel)$, we define $\del$-closed, $\adel$-closed, $\del$-exact and $\adel$-exact elements analogously.

\begin{defn}\label{defn 4.3.5}
An $H$-covariant complex $(B,\,d)$ is called an $H$-covariant differential graded algebra if $d$ satisfies
the graded Leibniz rule
\begin{equation}\label{eq 4.3.2}
d(bb')=d(b)b' + (-1)^kbd(b') \qquad b \in B^k, \quad b' \in B.
\end{equation}
\end{defn}

\begin{defn}\label{defn 4.3.6}
An $H$-covariant differential calculus over an $H$-module algebra $B$ (with unit map $\io_B$) is an $H$-covariant differential graded algebra $(\O,\,d)$ (with unit map $\io_{\O}$) such that $\O^0=B$, the two $H$-action on $B$ coming from $B$ itself and $\O^0$ coincide, and 
\begin{equation}\label{eq 4.3.3}
\O^k\,=\,{\rm span}_{\bC}\{b_0db_1\wdg\dots \wdg db_k \mid b_0,\dots,b_k \in B\}.
\end{equation}
\end{defn}

Observe that the coincidence of the two $H$-actions on $B$ implies that the two unit maps also coincide. Observe also that the induced $(A_l,A_l)$-bimodule structure on $\O$ coincide with the one coming from the unit map.

\begin{defn}\label{defn 4.3.7}
An $H$-covariant differential calculus $(\O,\,d)$ over an $H$-module $\*$-algebra $B$ is a $\*$-differential calculus if the involution of $B$ extends to a degree zero involutive conjugate linear map on $\O$, for which $(d\w)^*=d(\w^*)$ for all $\w \in \O$, and \[(\w \wdg \e)^*=(-1)^{kl}\e^* \wdg \w^*, \qquad \w \in \O^k, \quad \e \in \O^l\] making $\O$ an $H$-module $\*$-algebra. 
\end{defn}

Recall the Connes-Moscovici Hopf algebroid from Theorem \ref{thm 4.5.1}. We then have

\begin{prop}\label{prop 4.5.4}
Let $(\O,\,d)$ be a $Q$-covariant differential calculus on $A$, $Q$ being a Hopf algebra. Then $(\O,\,d)$ can be made into an $H$-covariant differential calculus on $A$ in the sense of Definition \ref{defn 4.3.6}. Furthermore, if $Q$ is a Hopf $\*$-algebra, $A$ is a $Q$-module $\*$-algebra and $(\O,\,d)$ is a $Q$-covariant $\*$-differential calculus then it can be made into an $H$-covariant $\*$-differential calculus in the sense of Definition \ref{defn 4.3.7}.
\end{prop}
	
\begin{proof}
We define the $H$-action on $\O$ as follows:
\begin{equation}\label{eq 4.5.8}
(a \tens_{\bC} q \tens_{\bC} b)\cdot \w=a(q \cdot \w)b.
\end{equation} The only non-trivial part to check is that \eqref{eq 4.3.1} holds. This is easy because $H_0$ contains $1 \tens_{\bC} Q \tens_{\bC} 1$.
\end{proof}


\begin{lem}\label{lem 4.3.8}
For an $H$-covariant $\*$-differential calculus $(\O,\,d)$, we have
\begin{enumerate}[i)]
\item $[h-s_l\ve_l(h),d]=0$ $\implies$ $[S^{-1}(h^*)-t_l\ve_l(S^{-1}(h^*)),d]=0$;
\item $[h-t_l\ve_l(h),d]=0$ $\implies$ $[S^{-1}(h^*)-s_l\ve_l(S^{-1}(h^*)),d]=0$;
\end{enumerate} for $h \in H$. Thus combining the two, we get that $h \in H_0$ if and only if $S^{-1}(h^*) \in H_0$.
\end{lem}

\begin{proof}
For $\w \in \O$, we compute 
\[
\begin{aligned}
0=([h-s_l\ve_l(h),d](\w^*))^*&=((h-s_l\ve_l(h)) \cdot d(\w^*)-d((h-s_l\ve_l(h)) \cdot \w^*))^*\\
&=((h-s_l\ve_l(h)) \cdot (d\w)^*)^*-d(((h-s_l\ve_l(h)) \cdot \w^*)^*)\\
&=(S(h-s_l\ve_l(h)))^*\cdot d\w-d(S(h-s_l\ve_l(h))^* \cdot \w)\\
&=[(S(h-s_l\ve_l(h)))^*,d](\w).
\end{aligned}
\] And similarly, $0=([h-t_l\ve_l(h),d](\w^*))^*=[(S(h-t_l\ve_l(h)))^*,d](\w)$. Now
\[
\begin{aligned}
(S(h-s_l\ve_l(h)))^*=S(h)^*-S(s_l\ve_l(h))^*&=S^{-1}(h^*)-S^{-1}(s_l(\ve_l(h)^*))\\
&=S^{-1}(h^*)-t_l\ve_l(S^{-1}(h^*))\\
\end{aligned}
\] and
\[
\begin{aligned}
(S(h-t_l\ve_l(h)))^*=S(h)^*-S(t_l\ve_l(h))^*&=S^{-1}(h^*)-s_l(\ve_l(h)^*)\\
&=S^{-1}(h^*)-s_l\ve_l(S^{-1}(h^*)).
\end{aligned}
\] Thus we get $S^{-1}(h^*) \in H_0$ if $h \in H_0$. The other direction follows from $(S \*)^2=\id$
\end{proof}

\begin{lem}\label{lem 4.3.9}
On $\ol{\O}$, defining the product as $\ol{\w} \wdg \ol{\e}=(-1)^{kl}\ol{\e \wdg \w}$ for $\w \in \O^k$, $\e \in \O^l$ makes $(\ol{\O},\ol{d})$ an $H$-covariant differential graded algebra. Then an $H$-covariant $\*$-differential calculus is an $H$-covariant differential calculus such that $\# : (\O,\,d) \to (\ol{\O},\ol{d})$ is $H$-linear and a differential graded algebra homomorphism.
\end{lem}

\begin{proof}
The second part follows from the discussion prior to Lemma \ref{lem 4.2.22}. For the first part, we observe that given $\w \in \O$
and $h \in H$,
\[
\begin{aligned}
{}[h-s_l\ve_l(h),\ol{d}](\ol{\w})&=(h-s_l\ve_l(h)) \cdot \ol{d}(\ol{\w})-\ol{d}((h-s_l\ve_l(h))\cdot \ol{\w})\\
&=(h-s_l\ve_l(h)) \cdot \ol{d\w}-\ol{d}(\ol{(S(h-s_l\ve_l(h)))^*\cdot \w})\\
&=\ol{S(h-s_l\ve_l(h))^*\cdot d\w}-\ol{d(S(h-s_l\ve_l(h))^*\cdot \w)}\\
&=\ol{[S(h-s_l\ve_l(h))^*,d](\w)}
\end{aligned}
\] and similarly, $[h-t_l\ve_l(h),\ol{d}](\ol{\w})=\ol{[S(h-t_l\ve_l(h))^*,d](\w)}$. Now the lemma follows from Lemma \ref{lem 4.3.8}.
\end{proof}

\begin{defn}\label{defn 4.3.10}
We define the space of invariant forms $\O_0$ of $\O$ as \[\O_0=\{\w \in \O \mid h \cdot \w =s_l\ve_l(h) \cdot \w=t_l\ve_l(h) \cdot \w \text{ for all } h \in H_0\}.\]
\end{defn}

Observe that we recover the usual definition of invariant subalgebra as in Lemma \ref{lem 4.2.22} if the differential $d$ is identically $0$.

\begin{prop}\label{prop 4.3.11}
For the space of invariant forms we have,
\begin{enumerate}[i)]
\item $(\O_0,d|_{\O_0})$ is a differential graded algebra;
\item $\O_0$ is a $\*$-algebra;
\item $d|_{\O_0}$ satisfies $d|_{\O_0}(\w^*)=(d|_{\O_0}\w)^*$ for all $\w \in \O_0$;
\item $\# : (\O_0,d|_{\O_0}) \to (\ol{\O}_0,\ol{d}|_{\ol{\O}_0})$ is a differential graded algebra homomorphism.
\end{enumerate}
\end{prop}

\begin{proof}
i) That $\O_0$ is an algebra follows from the same proof as in $d$ identically $0$ case.
Moreover, that $d$ preserves $\O_0$ follows from the definition of $H_0$.
	
ii) Observe that for $h \in H_0$ and $\w \in \O_0$
\[
\begin{aligned}
(h\cdot \w^*)^*=S(h)^* \cdot \w&=S^{-1}(h^*) \cdot \w\\
&=t_l\ve_l(S^{-1}(h^*)) \cdot \w\\
&=t_l(\ve_l(h)^*)\cdot \w\\
&=(Ss_l\ve_l(h))^* \cdot \w
\end{aligned}
\] so that \[h \cdot \w^*=((Ss_l\ve_l(h))^*\cdot \w)^*=s_l\ve_l(h)\cdot \w^*.\] Again
\[
\begin{aligned}
(h\cdot \w^*)^*=S(h)^* \cdot \w&=S^{-1}(h^*) \cdot \w\\
&=s_l\ve_l(S^{-1}(h^*)) \cdot \w\\
&=s_l(\ve_l(h)^*)\cdot \w\\
&=(s_l\ve_l(h))^* \cdot \w
\end{aligned}
\] so that \[h \cdot \w^*=((s_l\ve_l(h))^*\cdot \w)^*=S^{-1}s_l\ve_l(h)\cdot \w^*=t_l\ve_l(h)\cdot \w^*.\]
	
iii) holds because $d$ satisfies the property.
	
iv) Follows from ii).
\end{proof}

We shall denote the differential on $\O_0$ only by $d$, assuming that it really means $d$ is restricted to $\O_0$. 

\begin{defn}\label{defn 4.3.17}\mbox{}
\begin{enumerate}[i)]
\item We say that an $H$-covariant differential calculus $(\O,\,d)$ over an $H$-module algebra $B$ has total dimension $n$ if $\O^k=0$, for all $k > n$, and $\O^n \neq 0$. 
		
\item If in addition, there exists a $(B,B)$-bimodule and an $H$-module isomorphism $\vol : \O^n \to B$, then we say that $\O$ is orientable. 
		
\item If $\O$ is a $\*$-calculus over a $\*$-algebra, then a $\*$-orientation is an 
orientation which is also $\*$-preserving, meaning $\ol{\vol} \#=\# \vol$.
		
\item A $\*$-orientable calculus is one which admits a $\*$-orientation.
		
\item Let $\t$ be a state on $B$, i.e., a unital linear functional $\t : B \to \bC$ such that 
$\t(b^*b) \geq 0$. We call the functional $\t \vol$ the integral associated to $\t$ and denote 
it by $\int_{\t}$.
		
\item We say that the integral is closed if $\int_{\t}(d\w)=0$ for all $\w \in \O^{n-1}$.
\end{enumerate}
\end{defn}




\begin{lem}\label{lem 4.3.19}
Assume that $(\O,\,d)$ is $\*$-oriented with orientation $\vol$ and of total dimension $2n$. Then $(\O_0,\,d)$ is $\*$-oriented. 
\end{lem}

\begin{proof}
Since $\vol$ is assumed to be $H$-linear, it restricts to $\O_0$, which in turn shows
that $\O^{2n}_H \neq 0$ so that it also has total dimension $2n$. The
lemma now follows from Lemma \ref{lem 4.2.22} and Proposition \ref{prop 4.3.11}.
\end{proof}

\subsubsection{Covariant complex structures} As mentioned before, the setup below is due to \cite{MR3720811} and we follow it closely. We
shall omit the proofs of some of the results here as they are essentially given in \cite{MR3720811}.

\begin{defn}\label{defn 4.3.21}
An $H$-covariant almost complex structure for an $H$-covariant $\*$-differential calculus $(\O,\,d)$ over an $H$-module $\*$-algebra $B$ is an $\bN_0^2$-algebra grading $\dsum_{(k,l) \in \bN_0^2}\O^{(k,l)}$ for $\O$ such that 
\begin{enumerate}[i)]
\item the $H$-action preserves the $\bN_0^2$-grading;
\item $\O^n=\dsum_{k+l=n}\O^{(k,l)}$, for all $n \in \bN_0$;
\item $\# : \O \to \ol{\O}$ preserves the $\bN_0^2$-grading, where the $\bN_0^2$-grading on $\ol{\O}$ is given by $\ol{\O}^{(k,l)}=\ol{\O^{(l,k)}}$.
\end{enumerate}
\end{defn}

Let $\del$ and $\adel$ be the unique homogeneous operators of order $(1,0)$ and $(0,1)$ respectively, defined by 
\begin{equation}\label{eq 4.3.5}\del\mid_{\O^{(k,l)}}=\proj_{\O^{(k+1,l)}}  d \qquad \adel\mid_{\O^{(k,l)}}=\proj_{\O^{(k,l+1)}}  d,
\end{equation} where $\proj_{\O^{(k,l+1)}}$ and $\proj_{\O^{(k,l+1)}}$ are the projections from $\O^{(k+l+1)}$ onto $\O^{(k+1,l)}$ and $\O^{(k,l+1)}$, respectively.

\medskip\noindent

As in \cite{MR3720811}, we have:

\begin{lem}\label{lem 4.3.22}
If $\dsum_{(k,l) \in \bN_0^2} \O^{(k,l)}$ is an $H$-covariant almost complex structure for an
$H$-covariant $\*$-differential calculus $(\O,\,d)$ over an $H$-module $\*$-algebra $B$, then the
following two conditions are equivalent:
\begin{enumerate}[i)]
\item $d=\del+\adel$;
\item the triple $(\dsum_{(k,l) \in \bN_0^2} \O^{(k,l)},\del,\adel)$ is an H-covariant double complex.
\end{enumerate}
\end{lem}

\begin{proof}
The proof of the equivalence is in \cite{MR3720811}. All we have to show is the $H$-covariant part in ii). Observe that $\proj_{\O^{(k+1,l)}}$ and $\proj_{\O^{(k,l+1)}}$ are $H$-linear.
Then for $h \in H_0$, $$[h-s_l\ve_l(h),\del\mid_{\O^{(k,l)}}]=[h-t_l\ve_l(h),\del\mid_{\O^{(k,l)}}]
=0\, .$$ Thus we get \eqref{eq 4.3.1} for $\del\mid_{\O^{(k,l)}}$, and similarly for
$\adel\mid_{\O^{(k,l)}}$, hence the covariance.
\end{proof}

\begin{defn}\label{defn 4.3.23}
When the conditions in Lemma \ref{lem 4.3.22} hold for an almost complex structure, then we say that
the almost complex structure is integrable.
\end{defn}

We also call an integrable almost complex structure a complex structure and the double complex $(\dsum_{(k,l) \in \bN_0^2} \O^{(k,l)},\del,\adel)$ its Dolbeault double complex. Note that 
\begin{equation}\label{eq 4.3.6}
\del(\w^*)=(\adel \w)^*, \qquad \adel(\w^*)=(\del \w)^*, \qquad \w \in \O,
\end{equation}
as they follow from the integrability condition.

\begin{lem}\label{lem 4.3.24}
Suppose that $(\O,\,d)$ admits an $H$-covariant complex structure. Then $(\O_0,\,d)$ admits a complex structure. We call this a transverse complex structure on $B_0$.
\end{lem} 


\begin{proof}
Condition i) in Definition \ref{defn 4.3.21} implies that $(\O_0,\,d)$ admits an $\bN^2_0$-algebra grading by $(\O_0)^{(k,l)}=\O^{(k,l)}_0$, $(k,l) \in \bN^2_0$. Condition ii) follows automatically, while Condition iii) follows from that fact that $\#$ is $H$-linear. $\del$ and $\adel$ 
restrict to the space of invariant forms as in Proposition \ref{prop 4.3.11}. Finally, $d=\del+\adel$ then follows automatically.
\end{proof}

\subsubsection{Covariant hermitian and K\"ahler structures} \label{subsec 4.3.3}

We fix an $H$-covariant $\*$-differential calculus $(\O,\,d)$ over an $H$-module $\*$-algebra $B$ of total dimension $2n$. As in \cite{MR3720811}, the following is a non-commutative generalization of an almost symplectic form.

\begin{defn}\label{defn 4.3.30}
An almost symplectic form for $\O$ is a central real $H$-invariant 2-form $\s$ ($h \cdot \s=
s_l\ve_l(h) \cdot \s$ for all $h \in H$) such that, the Lefschetz operator
\[L : \O \to \O , \qquad \w \mto \s \wdg \w\] satisfies the following condition:
the maps
\begin{equation}\label{eq 4.3.8}
L^{n-k} : \O^k \to \O^{2n-k}
\end{equation}
are isomorphisms for all $0 \leq k < n$.
\end{defn}

Since $\s$ is a central real form, $L$ is a $(B,B)$-bimodule morphism and $\*$-preserving
$(\ol{L}\#=\# L)$. Moreover, the $H$-invariance condition implies that $L$ is also an $H$-module
morphism. Indeed, we have $$h \cdot (\s \wdg \w)=h_{(1)} \cdot \s \wdg h_{(2)} \cdot \w=
\io_B(\ve_l(h_{(1)})) \s \wdg h_{(2)} \cdot \w
$$
$$
=\s \wdg \io_B(\ve_l(h_{(1)}))(h_{(2)} \cdot \w))=\s \wdg (s_l\ve_l(h_{(1)})h_{(2)}) \cdot \w=\s \wdg h \cdot \w\, .$$

\begin{defn}\label{defn 4.3.31}
A symplectic form is a $d$-closed almost symplectic form.
\end{defn}

\'O Buachalla, \cite{MR3720811}, introduced hermitian structure which is an almost symplectic form 
compatible with a complex structure.

\begin{defn}\label{defn 4.3.32}
A hermitian structure for $\O$ is a pair $(\O^{(\cdot,\cdot)},\s)$, where
$\O^{(\cdot,\cdot)}$ is an $H$-covariant complex structure, and $\s$ is an almost symplectic form, called the hermitian form, such that $\s \in \O^{(1,1)}$.
\end{defn}

We have:

\begin{lem}\label{lem 4.3.33}
Suppose that $(\O^{(\cdot,\cdot)},\s)$ is a hermitian structure for $(\O,\,d)$. Then $\s$ induces a hermitian structure on $(\O_0,\,d)$.
\end{lem}

\begin{proof}
By definition, $\s \in \O_0$. The $H$-linearity of $L$ shows that $\s$ is an almost symplectic form for $(\O_0,\,d)$. Finally, $\s \in (\O^{(1,1)})_0=(\O_0)^{(1,1)}$, by Lemma \ref{lem 4.3.24}.
\end{proof}

We say that an almost complex structure is of diamond type if $\O^{(a,b)}=0$ whenever $a>n$ or $b>n$. Supposing $a>n$ and observing that the isomorphism $L^{a+b-n}$ maps $\O^{(n-b,n-a)}$ onto $\O^{(a,b)}$, we see that
the existence of a hermitian structure implies that the complex structure has to be of diamond type.

\begin{defn}\label{defn 4.3.34}
The Hodge map associated to a hermitian structure is the morphism uniquely defined by 
\begin{equation}\label{eq 4.3.9}
\star(L^j(\w))=(-1)^{\frac{k(k+1)}{2}}i^{a-b}\frac{[j]!}{[n-j-k]!}L^{n-j-k}(\w) \qquad \w \in P^{(a,b)} \subset P^k.
\end{equation}
\end{defn}

Recall the notion of primitive forms from Definition \ref{defn 3.2.5}. Observe that $\star$ is an $H$-module morphism. Hence it descends to $\O_0$.

\begin{lem}\label{lem 4.3.35}
We have
\begin{enumerate}[i)]
\item $\star^2(\w)=(-1)^k\w$ for all $\w \in \O^k$,
\item $\star$ is an isomorphism,
\item $\star(\O^{(a,b)})=\O^{(n-b,n-a)}$,
\item $\star$ is a $\*$-preserving.
\end{enumerate}
\end{lem}


Given a hermitian structure $(\O^{(\cdot,\cdot)}, \k)$, we first recover the hermitian metric 
associated to it:

\begin{defn}\label{defn 4.3.36}
Define $g : \O \tens_B \ol{\O} \to B$ by $g(\w \tens \ol{\e})=0$ for $\w \in \O^k$, $\e \in 
\O^l$, $k\neq l$, and
\begin{equation}\label{eq 4.3.10}
g(\w \tens \ol{\e})=\vol(\w \wdg \star(\e^*))
\end{equation} 
\end{defn} for $\w, \e \in \O^k$. 

A metric on the orbit space should be an invariant one as is showed in the following lemma.

\begin{lem}\label{lem 4.3.37}
For $\w, \e \in \O^k$ and $h \in H$, it holds that
\begin{equation}\label{eq 4.3.11}
g(h_{(1)} \cdot \w \tens h_{(2)} \cdot \ol{\e})=h \cdot g(\w \tens \ol{\e}),
\end{equation} so that $g$ is $H$-covariant.
\end{lem}

\begin{proof}
We compute
\[
\begin{aligned}
g(h_{(1)} \cdot \w \tens h_{(2)} \cdot \ol{\e})&=g(h_{(1)} \cdot \w \tens \ol{S(h_{(2)})^* \cdot \e})\\
&=\vol(h_{(1)} \cdot \w \wdg \star(S(h_{(2)})^* \cdot \e)^*)\\
&=\vol(h_{(1)} \cdot \w \wdg \star((S(S(h_{(2)})^*))^*) \cdot \e^*)\\
&=\vol(h_{(1)} \cdot \w \wdg \star(h_{(2)} \cdot \e^*)\\
&=\vol(h_{(1)} \cdot \w \wdg h_{(2)} \cdot \star(\e^*))\\
&=\vol(h \cdot (\w \wdg \star(\e^*)))\\
&=h \cdot \vol(\w \wdg \star(\e^*))\\
&=h \cdot g(\w \tens \ol{\e}),
\end{aligned}
\]
\end{proof}

\begin{prop}\label{prop 4.3.38}
The following decompositions are orthogonal with respect to $\<,\>$:
\begin{enumerate}[i)]
\item The degree decomposition $\O=\dsum_{k} \O^k$;
\item The bidegree decomposition $\O^k=\dsum_{(a,b)}\O^{(a,b)}$;
\item The Lefschetz decomposition $\O^k=\dsum_{j\geq 0} L^j(P^{k-2j})$.
\end{enumerate}
\end{prop}

The above proposition implies the following.

\begin{cor}\label{cor 4.3.39}
We have $g(\w \tens \ol{\e})=g(\e \tens \ol{\w})^*$ for $\w, \e \in \O$.
\end{cor}

The hermitian structure is said to be positive definite if $g(\w \tens \ol{\w})>0$ for all nonzero $\w \in \O$. In that case, we define an inner product (positive definite, hermitian) on $\O$ by setting
\begin{equation}\label{eq 4.3.12}
\< \w,\e \>=\t g(\w \tens \ol{\e})=\int_{\t}\w \wdg \star(\e^*)
\end{equation} for $\w, \e \in \O$ and a fixed faithful state $\t$ on $B$. We denote the
corresponding norm of $\w$ by $\|\w\|$. Moreover, Lemma \ref{lem 4.3.37} shows that $g$ induces
a metric on $\O_0$ that takes values in $B_0$. Applying $\t$, we get an inner product
on $\O_0$ which is really the restriction of $\<\cdot,\cdot\>$ to $\O_0$. From now on, we
assume that the hermitian structure to be positive definite.

\begin{prop}\label{prop 4.3.45}
The Hodge map $\star$ is unitary.
\end{prop}

\begin{proof}
See the proof of Lemma 5.10 in \cite{MR3720811}.
\end{proof}

We now define the Laplacians.

\begin{defn}\label{defn 4.3.46}\mbox{}
\begin{enumerate}[i)]
\item The codifferential is defined as $d^*:=-\star d\star$;
\item the holomorphic codifferential is defined as $\del^*:=-\star\adel\star$;
\item the anti-holomorphic codifferential is defined as $\adel^*=-\star\del\star$.
\end{enumerate}
\end{defn}

Observe that for $\w \in \O$, \begin{equation}\label{eq 4.3.13}d^*(\w^*)=(d^*\w)^*, \qquad \del^*(\w^*)=(\adel^*\w)^* \qquad \text{and} \qquad \adel^*(\w^*)=(\del^*\w)^*.\end{equation} Now, it is natural to define the $d$-, $\del$- and $\adel$- Laplacians, respectively as \begin{equation}\label{eq 4.3.14}\D_d:=(d+d^*)^2, \qquad \D_{\del}:=(\del+\del^*)^2, \qquad \D_{\adel}:=(\adel+\adel^*)^2.\end{equation}

\begin{prop}\label{prop 4.3.47}
The operator adjoints of $d$, $\del$ and $\adel$ are $d^*$, $\del^*$ and $\adel^*$, respectively.
\end{prop}

The following will be used later.

\begin{cor}\label{cor 4.3.48}
The Laplacians $\D_d$, $\D_{\del}$ and $\D_{\adel}$ are symmetric.
\end{cor}

We have:

\begin{lem}\label{lem 4.3.49}
The operator $d^*$ (respectively $\del$, $\adel$) and hence $\D_d$ (respectively $\D_{\del}$,
$\D_{\adel}$) descends to $\O_0$.
\end{lem}

\begin{proof}
Since $\star$ is $H$-linear, we have for $h \in H_0$,
$$[h-s_l\ve_l(h),d^*]=[h-t_l\ve_l(h),d^*]=0\, .$$ Hence $d^*$ descends to $\O_0$.
\end{proof}




According to \cite{MR3720811}, K\"ahler structures are defined as follows.

\begin{defn}\label{defn 4.3.51}
A K\"{a}hler structure is a hermitian structure $(\O^{(\cdot,\cdot)}, \k)$ such that the hermitian form $\k$ is $d$-closed. Such a form is called a K\"{a}hler form. 
\end{defn}



		



\begin{lem}\label{lem 4.3.55}
A K\"{a}hler structure $(\O^{(\cdot,\cdot)}, \k)$ on $(\O,\,d)$ induces via $\k$ a K\"ahler structure on $(\O_0,\,d)$.
\end{lem}

\begin{proof}
Since $\k$ is automatically $d|_{\O_0}$-closed, the lemma follows from Lemma \ref{lem 4.3.33}.
\end{proof}










\subsection{The main example - Transversely K\"ahler foliations}
In this subsection, we show how transversely K\"ahler foliations fit into our framework.

\medskip\noindent

According to Haefliger \cite{MR 2456103}:

\begin{defn}\label{defn 4.3.12}
A transverse structure on a foliated manifold $(M,\cF)$ is a structure on the transversal manifold $N$, invariant under the action of the holonomy pseudogroup $P$. 
\end{defn}

Since the groupoid $\G(P)$ is constructed out of $P$, it follows
that $P$ invariant structures are $\G(P)$ invariant. The normal bundle $N(M,\cF)$ of the foliation $\cF$ is isomorphic to the tangent bundle $TN$ of $N$. Thus, basic forms on the foliated manifold $(M,\cF)$ are in bijective correspondence with $\G(P)$-invariant
forms on the transverse manifold $N$ (see \cite{MR 2456103}). To see what does $\G(P)$ invariant forms correspond to, we introduce 
the following.

\begin{defn}\label{defn 4.3.13}
A local bisection of a Lie groupoid $G$ is a local section $\s : U \to G$ of $s : G \to G_0$ defined on an open subset $U \subset G_0$ such that $t\s$ is an open embedding. 
\end{defn}
If $G$ is \'etale, any arrow $g$ induces a germ of a homeomorphism $\s_g : (U,s(g)) \to (V,t(g))$ from a neighborhood $U$ of $s(g)$ to a neighborhood $V$ of $t(g)$ as follows: choosing $U$ small enough such that a bisection $\s$ exists and $t|_{\s U}$ is a homeomorphism into $V:=t(\s U)$, we set $\s_g:=t\s$. We do not distinguish between $\s_g$ and the actual germ of this map at the point $s(g)$.

\begin{lem}\label{lem 4.3.14}
Let $G$ be an \'etale groupoid, and let $E$ be a smooth complex vector bundle over $G_0$ with a $G$-representation. Then a section $u : G_0 \to E$ is $G$-invariant if and only if it is $C_c^{\infty}(G)$-invariant.
\end{lem}

\begin{proof}
Recall that a section $u$ of the bundle $E$ is $G$ invariant, if $g \cdot u(x)=u(y)$ for all arrow $x 
\xrightarrow{g} y$, while $u$ is $C_c^{\infty}(G)$ invariant if $a \cdot
u=\ve_l(a)u$ for all $a \in C_c^{\infty}(G)$. That $G$-invariance implies $C_c^{\infty}(G)$-invariance is
clear. For the converse, pick an 
arrow $x \xrightarrow{g} y$ and a bisection $(U,\s)$ such that $g \in \s(U)$ \cite{MR 2012261}. Then choose 
any function $a \in C_c^{\infty}(G)$ with support in $\s(U)$ and $a(g)=1$. Note that on a bisection $\s(U)$, we have
$a(t|_{\s(U)})^{-1}=\ve_l(a)$ and $a \cdot u=a(t|_{\s(U)})^{-1}u=\ve_l(a)u$. Hence the lemma follows.
\end{proof}

Now take $B=C^{\infty}(G_0)$ and $\O=\O(G_0)$, the $\bC$-valued smooth functions and forms on $G_0$, respectively. 

\begin{lem}\label{lem 4.3.15}
The differential $d$ on $G_0$ satisfies 
\begin{equation}\label{eq 4.3.4}
d(a \cdot \w)=d(\ve_l(a)) \wdg \w + a \cdot d(\w)
\end{equation} for $a \in C_c^{\infty}(G)$ and $\w \in \O(G_0)$. Hence $[a-\ve_l(a),d]=0$ for all $a \in C_c^{\infty}(G)$, thus
implying $H_0=C_c^{\infty}(G)$ (see \eqref{eq 4.3.1n} for $H_0$).
\end{lem}

\begin{proof}
As observed above in the proof of Lemma \ref{lem 4.3.14}, on a bisection $\s(U)$, we have $a(t|_{\s(U)})^{-1}=\ve_l(a)$ and $a \cdot u=a(t|_{\s(U)})^{-1}u=\ve_l(a)u$. Now \eqref{eq 4.3.4} follows from Leibniz rule and locality of $d$. The last statement follows from \eqref{eq 4.3.4} and the fact that $s_l\equiv t_l$.
\end{proof}

Denote by $\O(G_0)^{G}$ the $G$-invariant forms. Then forms on the ``orbit or leaf space" are captured as follows.

\begin{prop}\label{prop 4.3.16}
The pair
$(\O(G_0),\,d)$ is a $C_c^{\infty}(G)$-covariant differential calculus, and we have $(\O(G_0)^G,\,d)=(\O(G_0)_{C_c^{\infty}(G)},\,d)$ as differential graded algebras.
\end{prop}

\begin{proof}
Since $G$ acts by local diffeomorphisms, it follows that $d$ is $G$-invariant. So $d$ descends to $\O(G_0)^G$. The
proposition now follows from Lemma \ref{lem 4.3.14} and Lemma \ref{lem 4.3.15}.
\end{proof}

		
		
		
		
		

\begin{defn}\label{defn 4.3.18}
An \'etale groupoid $G$ is oriented if $G_0$ is oriented in the ordinary sense and $G$ acts by orientation-preserving local diffeomorphisms.
\end{defn}

\begin{prop}\label{prop 4.3.19}
With $B=C^{\infty}(G_0)$ and $\O=\O(G_0)$, orientation in the sense of Definition \ref{defn 4.3.17} coincide with groupoid orientation on $G$.
\end{prop}

\begin{proof}
This follows from Proposition \ref{prop 4.3.16}.
\end{proof}

As in \cite{MR1151583}, we define:

\begin{defn}\label{defn 4.3.26}
The foliation $\cF$ on a foliated manifold $(M,\cF)$ is transversely holomorphic if it carries a transverse complex structure in the sense of Definition \ref{defn 4.3.12}.
\end{defn}

If the foliation $\cF$ is transversely holomorphic, the normal bundle $N(M,\cF)$ of $\cF$ has a complex structure corresponding to the complex structure on $N$. Therefore any complex-valued basic $k$-form can be represented as a sum of the $k$-forms of pure type $(r,s)$ corresponding to the decomposition of $k$-forms on the complex manifold $N$. Let $\O^k_{\bC}(M,\cF)$ denote
the space of complex-valued basic $k$-forms on the foliated manifold $(M,\cF)$,
and denote by $\O^{(r,s)}_{\bC}(M,\cF)$ the space of complex-valued basic forms of pure
type $(r,s)$. Then we have $\O^k_{\bC}(M,\cF)=\dsum_{r+s=k}\O^{(r,s)}_{\bC}(M,\cF)$. The exterior derivative $d : \O^k_{\bC}(M,\cF) \to \O^{k+1}_{\bC}(M,\cF)$ decomposes into two components $d=\del+\adel$, where $\del$ is of bidegree $(1,0)$ and $\adel$ is of bidegree $(0,1)$, i.e., $\del : \O^{(r,s)} \to \O^{(r+1,s)}$ and $\adel : \O^{(r,s)} \to \O^{(r,s+1)}$. 

\medskip\noindent

Keeping in mind Definition \ref{defn 4.3.26} and the case for orbifolds (see \cite{MR3751961}), we make

\begin{defn}\label{defn 4.3.27}
An \'etale groupoid $G$ is holomorphic if $G_0$ is a complex manifold and $G$ acts by local biholomorphic transformations.
\end{defn}

This fits into our framework as follows:

\begin{prop}\label{prop 4.3.28}
An \'etale groupoid $G$ is holomorphic if and only if $(\O(G_0),\,d)$ admits a $C_c^{\infty}(G)$-covariant complex structure.
\end{prop}

\begin{proof}
First observe that an almost complex structure on $G_0$ is also given by a bundle map $J : T^*(G_0) \to 
T^*(G_0)$ (and its extension to the exterior algebra bundle) such that
$J^2=-\id_{T^*(G_0)}$. The bidegree 
decomposition is a consequence of this fact. Since bundle maps are sections of the $\Hom$-bundle, $G$ is 
almost complex if and only if $(\O(G_0),\,d)$ admits a $C_c^{\infty}(G)$-covariant almost complex structure, by 
Lemma \ref{lem 4.3.14}. Since integrability is same in both sense, we have the proposition proved. 
\end{proof}

The orbit space inherits a complex structure:

\begin{cor}\label{cor 4.3.29}
If $G$ is holomorphic, then $(\O(G_0)^G,\,d)$ admits a complex structure.
\end{cor}

\begin{proof}
This follows from Proposition \ref{prop 4.3.28} and Lemma \ref{lem 4.3.24}.
\end{proof}

We recall from \cite{MR2012261}:

\begin{defn}\label{defn 4.3.40}
The foliation $\cF$ on a foliated manifold $(M,\cF)$ is transversely Riemannian if it carries a transverse Riemannian structure in the sense of Definition \ref{defn 4.3.12}.
\end{defn}

The metric on $N(M,\cF)$ is induced from a bundle-like metric on $M$. Recall from \cite{MR1151583}:

\begin{defn}\label{defn 4.3.41}
The foliation $\cF$ on a foliated manifold $(M,\cF)$ is transversely hermitian if it carries a transverse hermitian structure in the sense of Definition \ref{defn 4.3.12}.
\end{defn}

The operator $\star : \L^k(M,\cF) \to \L^{2q-k}(M,\cF)$ defined via the transverse part of the bundle-like metric of $\cF$ extends to $\L^k_{\bC}(M,\cF) \to \L^{2q-k}_{\bC}(M,\cF)$, where $q$ is the complex codimension of $\cF$.

Being motivated by this, we make the following definition.

\begin{defn}\label{defn 4.3.42}
An \'etale groupoid $G$ is hermitian if $G_0$ admits a $G$-invariant hermitian structure.
\end{defn}

Again, algebraically we have the following proposition.

\begin{prop}\label{prop 4.3.43}
An \'etale groupoid $G$ is hermitian if and only if $(\O(G_0),\,d)$ admits a
$C_c^{\infty}(G)$-covariant hermitian structure.
\end{prop}

\begin{proof} The proof of the statement that $G$ is hermitian implies that $(\O(G_0),\,d)$ admits a 
$C_c^{\infty}(G)$-invariant hermitian structure is straightforward. For the converse, we recover the 
hermitian metric as in Definition \ref{defn 4.3.36}, and Lemma \ref{lem 4.3.37} shows that it is 
$G$-invariant. Compatibility follows from Proposition \ref{prop 4.3.38}. \end{proof}

\begin{cor}\label{cor 4.3.44}
If $G$ is hermitian, then $(\O(G_0)^G,\,d)$ admits a hermitian structure.
\end{cor}

\begin{proof}
This follows from Proposition \ref{prop 4.3.43} and Lemma \ref{lem 4.3.33}.
\end{proof}

Following \cite{MR1151583}, we have:

\begin{defn}\label{defn 4.3.56}
The foliation $\cF$ on a foliated manifold $(M,\cF)$ is transversely K\"ahler if it carries a transverse K\"ahler structure in the sense of Definition \ref{defn 4.3.12}. 
\end{defn}

The K\"ahler form of $N$ defines a basic $(1,1)$-form on $(M,\cF)$ which is called the transverse K\"ahler form of the foliation $\cF$. Motivated by this and the case for orbifolds, we define:

\begin{defn}\label{defn 4.3.57}
An \'etale groupoid $G$ is K\"ahler if $G_0$ admits a $G$-invariant K\"ahler structure.
\end{defn}

The following is routine:

\begin{prop}\label{prop 4.3.58}
An \'etale groupoid $G$ is K\"ahler if and only if $(\O(G_0),\,d)$ admits a $C_c^{\infty}(G)$-covariant K\"ahler structure.
\end{prop}

\begin{proof}
This follows from Proposition \ref{prop 4.3.43} and Proposition \ref{prop 4.3.16}.
\end{proof}

\begin{cor}\label{cor 4.3.59}
If $G$ is K\"ahler, then $(\O(G_0)^G,\,d)$ admits a K\"ahler structure.
\end{cor}

\begin{proof}
This follows from Proposition \ref{prop 4.3.58} and Lemma \ref{lem 4.3.55}.
\end{proof}

\subsection{Left bialgebroid covariance of universal \texorpdfstring{$1$}{}-forms}\label{subsec 5.3.3}
Let $X$ be the finite set $\{1,\dots,n\}$. 
In this subsection, we sketch the constrution of a left bialgebroid over $C(X)$ whose action on $C(X)$ lifts to the space of universal one forms on $C(X)$, in the sense of Definition \ref{defn 4.3.3}. For that, let us consider the left bialgebroid $C(X \times X \times X \times X)$ over $C(X \times X)$, constructed in the same way as above. We identify $C(X)$ as the first copy of the product $C(X) \tens C(X)=C(X \times X)$. Let $f \in C(X)$ and $h \in C(X \times X \times X \times X)$. Then
\begin{equation}\label{eq 5.3.37}
\begin{aligned}
h \cdot (f \tens 1) (x,y)&=\sum_{z,w}h \* s(f \tens 1)(z,w,x,y)\\
&=\sum_{z,w,\a,\b}h(\a,\b,x,y)s(f \tens 1)(z,w,\a,\b)\\
&=\sum_{z,w,\a,\b}h(\a,\b,x,y)\d_{z,\a}\d_{w,\b}(f \tens 1)(z,w)\\
&=\sum_{z,w}h(z,w,x,y)f(z).
\end{aligned}
\end{equation} So a sufficient condition that $h$ takes $f \tens 1$ to an element of the same form is that $\sum_wh(z,w,x,y)$ does not depend on $y$. Therefore we can write 
\begin{equation}\label{eq 5.3.38}
\sum_w h(z,w,x,y)=\sum_w h(z,w,x,x).
\end{equation} Now the space of universal one forms can be identified with functions on $X \times X$ vanishing on the diagonal. If $\sum_i f_i \tens g_i$ is such an element then 
\begin{equation}\label{eq 5.3.39}
h \cdot (\sum_if_i \tens g_i)(x,y)=\sum_{i,z,w}h(z,w,x,y)f_i(z)g_i(w).
\end{equation} Thus a sufficient condition that $h$ preserves this space is that 
\begin{equation}\label{eq 5.3.40}
h(z,w,x,x)=0
\end{equation} for $z \neq w$. Observe that, \eqref{eq 5.3.40}
together with \eqref{eq 5.3.38}
imply 
\begin{equation}\label{eq 5.3.42}
\sum_wh(z,w,x,y)=h(z,z,x,x).
\end{equation} A sufficient condition on $h$ such that $[h-s\ve(h),d]=0$ holds is
\begin{equation}\label{eq 5.3.43}
\sum_zh(z,w,x,y)=h(w,w,y,y).
\end{equation} 
Let $H_0$ be the set of all $h \in C(X \times X \times X \times X)$ such that
\begin{enumerate}[i)]
\item $\sum_wh(z,w,x,y)=h(z,z,x,x);$
\medskip\noindent
\item $h(z,w,x,x)=0$ for $z \neq w;$
\medskip\noindent
\item $\sum_zh(z,w,x,y)=h(w,w,y,y);$
\medskip\noindent
\end{enumerate}

\medskip\noindent

And let $H$ be the smallest subalgebra of $C(X \times X \times X \times X)$ containing $C(X)$ and $H_0$. 

\begin{prop}\label{prop 5.3.6}
$H$ is a left bialgebroid over $C(X)$ such that the action on $C(X)$ lifts to an action on the space of universal one forms in the sense of Definition \ref{defn 4.3.3}. 
\end{prop}

	
	

\begin{conj}
	$H$ is not of the form $C(X)\# Q$ for any Hopf algebra $Q$ acting on $C(X)$.
\end{conj}

\begin{rem}
We remark that the space $H_0$ is at least $n^2$-dimensional and hence $H$ is at least $n^3$-dimensional, sufficiently large for our purposes. 
\end{rem}

\begin{que}
Give a nice characterization/description of $H$.
\end{que}

\noindent
Finally, we ask a question which is not directly related to this work but interesting in its own right. In \cite{MR3777415}, it is shown that a coaction of a compact quantum group on an algebra can be lifted to a differential calculus (at least in the classical situation) under some suitable (unitarity of the coaction, technically, see also \eqref{eq 4.3.10}) conditions, like one expects from a group action. So we ask

\begin{que}\label{que 4.6.9}
Is the above true for unitary action (i.e., \eqref{eq 4.3.10} is satisfied) of a Hopf algebroid?
\end{que}

\begin{rem}
We have computed the complex structures on a three-point space in \cite{three}, none of which are covariant with respect to the natural $S_3$-action. It would be interesting to investigate covariance of the Hopf algebroid constructed in this subsection.
\end{rem}

\subsection{Covariant complex structures on toral deformed manifolds}\label{toralsubsec}
In this subsection, we sketch the construction of a natural Hopf algebroid and a covariant complex structure on toral deformed manifolds a la Connes, Dubois-Violette, Landi, and Rieffel. This provides a rather large class of genuine (non-classical) examples of Hopf algebroid covriant complex structures.

\medskip
Let $(M,T)$ be a  toral manifold, by which we mean a compact manifold together with a smooth action of a torus $T$. We shall consider toral deformations of $M$ in the sense of \cites{MR1846904,MR1937657}, see also \cite{MR1184061} for the analytic aspects. Making repeated use of Lemma 7.18 and Lemma 7.19 of \cite{bgm}, we obtain

\begin{prop}
Let $H$ be a Hopf algebroid over $C^{\infty}(M)$ and let $T$ act on $H$ in such a way so that all the structure maps are $T$-equivariant. Then $H_{\theta}$ is a Hopf algebroid over $C^{\infty}(M)_{\theta}$.
\end{prop}

As a concrete example, let $\mathcal{D}(M)$ be the algebra of differential operators. It is well-known (see \cite{MR1815717}) that $\mathcal{D}(M)$ is a Hopf algebroid. By the Proposition above, we obtain

\begin{prop}
$\mathcal{D}(M)_{\theta}$ is a Hopf algebroid over $C^{\infty}(M)_{\theta}$.
\end{prop}

Since a differential operator is given by vector fields locally, and since Lie derivative with respect to a vector field commutes with the exterior derivative $d$, we see that the canonical action of $\mathcal{D}(M)$ on $\Omega^*(M)$ satisfies our definition, where we take $H_0$ (as in Definition \ref{defn 4.3.3}, Equation \ref{eq 4.3.1n}) to be the set of $L_X$, $X$ being a vector field. 

\begin{prop}
$\Omega^*(M)_{\theta}$ is a $\mathcal{D}(M)_{\theta}$-covariant differential calculus in the sense of Definition \ref{defn 4.3.6}.
\end{prop}

\begin{proof}
It follows from the general construction that $\Omega^*(M)_{\theta}$ is indeed a differential calculus over $C^{\infty}(M_{\theta})$. Now since the $T$-action keeps the space $H_0$ invariant, the result follows from a totality argument similar to Lemma 7.28 of \cite{bgm}.
\end{proof}

By considering complex conjugation, we see that $\Omega^*(M)_{\theta}$ is even a $\mathcal{D}(M)_{\theta}$-covariant $\ast$-differential calculus. Before introducing complex structure in this example, let us summarize the above facts as follows:

\begin{thm}
Let $(M,T)$ be a toral manifold and let $\mathcal{D}(M)$ be the algebra of differential operators on $M$. Then $\Omega^*(M)_{\theta}$ is a $\mathcal{D}(M)_{\theta}$-covariant $\ast$-differential calculus over $C^{\infty}(M)_{\theta}$.
\end{thm}

Now let $(M,T)$ be a complex toral manifold, by which we mean that $M$ is a complex manifold (with equivariant complex structure $J$), $T$ is a complex torus and the action $T \times M \rightarrow M$ is holomorphic. Let $\mathcal{D}^{\mathrm{hol}}(M)$ be the subalgebra of $\mathcal{D}(M)$ generated by $C^{\infty}(M)$ and differential operators $D$ such that $D(J)=0$. For example, if $X$ is a holomorphic vector field, then $L_{X}$ is such an operator. It is well-known that such a $D$ preserves the $(p,q)$-decomposition (the $\bN_0^2$-grading) of $\Omega^{\ast}(M)$. 

\begin{prop}
$\mathcal{D}^{\mathrm{hol}}(M)$ is a Hopf algebroid over $C^{\infty}(M)$ and the complex structure on $M$ is a $\mathcal{D}^{\mathrm{hol}}(M)$-covariant complex structure on $\Omega^*(M)$.
\end{prop}

\begin{rem}
The $\bN_0^2$-grading comes from the eigenspaces of $J$ as in \cite{MR3073899}.
\end{rem}

The holomorphicity of the toral action yields

\begin{prop}
$\mathcal{D}^{\mathrm{hol}}(M)_{\theta}$ is a Hopf algebroid over $C^{\infty}(M)_{\theta}$ and $J_{\theta}$ is a $\mathcal{D}^{\mathrm{hol}}(M)_{\theta}$-covariant complex structure on $\Omega^*(M)_{\theta}$.
\end{prop}

Summarizing all these,

\begin{thm}
Let $(M,T)$ be a complex toral manifold and let $\mathcal{D}^{\mathrm{hol}}(M)$ be the algebra of differential operators on $M$ defined as above. Then $\Omega^*(M)_{\theta}$ is a $\mathcal{D}^{\mathrm{hol}}(M)_{\theta}$-covariant $\ast$-differential calculus over $C^{\infty}(M)_{\theta}$ and $J_{\theta}$ is a $\mathcal{D}^{\mathrm{hol}}(M)_{\theta}$-covariant complex structure.
\end{thm}

As a concrete example, let us consider $\mathbb{T}^{2n}$ acting on itself by translation. Then

\begin{thm}
$\mathcal{D}^{\mathrm{hol}}(\mathbb{T}^{2n})_{\theta}$ is a Hopf algebroid over $\mathbb{T}^{2n}_{\theta}$ and $\Omega^*(\mathbb{T}^{2n})_{\theta}$ is a $\mathcal{D}^{\mathrm{hol}}(\mathbb{T}^{2n})_{\theta}$-covariant differential calculus over $\mathbb{T}^{2n}_{\theta}$ together with a $\mathcal{D}^{\mathrm{hol}}(\mathbb{T}^{2n})_{\theta}$-covariant complex structure.
\end{thm}

\section{Further directions and comments} \label{sec 4.6}

We end this paper by discussing some directions that we have not touched upon.

\subsection{Comparison with Connes' approach}\label{subsec 4.6.1} In \cites{MR679730, MR823176, MR866491}, the approach taken to study singular spaces, in particular, the leaf space of a foliation is as follows. One models the singular space by a groupoid $G$ and then considers the convolution algebra $C_c^{\infty}(G)$ as the function algebra of the space in question. We have considered the groupoid here also, but as symmetries. To consider noncommutative complex geometry on the singular space, we need a differential calculus on the algebra $C_c^{\infty}(G)$. Here
there are many choices and it is a priori not clear what is the correct choice to make. In fact, if one takes a discrete group and view it as a groupoid then the convolution algebra is the group algebra and we don't know what a choice of differential calculus would be (neither the universal one nor a bicovariant one), let alone the study of noncommutative complex structure and the meaning of it. So before moving onto arbitrary groupoids, one needs to answer the following question.

\begin{que}\label{que 4.6.1}
Construct (or even classify) differential calculi on the group algebra $\bC\G$ of a discrete group $\G$. Are there any complex structures on it? If so, what does it mean to have a complex structure on $\bC\G$? 
\end{que}

\subsection{Comparison with Fr\"ohlich et al.'s approach}\label{subsec 4.6.2} In \cite{MR1614993}, they study spectral data 
associated to hermitian, K\"ahler structure. \cite{MR3720811} already mentions this and it is 
being taken up by him and collaborators \cite{2019arXiv190307599D}. We sketch this in our set up. Note that $H$ is represented on $L^2(\O)$ by unbounded operators with common domain $\O$. We first show that these operators are closable, by exhibiting densely defined adjoint operators. Taking ideas from \cite{MR2817646}, we exploit the $(A_r,A_r)$-bimodule structure on $\O \tens_B \ol{\O}$ which is given by \eqref{eq 4.2.26} via $\ta^{-1} : A_r \to A_l^{op}$; explicitly, 
\begin{equation}\label{eq 4.6.1}
a_1 \cdot (\w \tens \ol{\e}) \cdot a_2=S(s_r(a_2)) \cdot \w \tens s_r(a_1) \cdot \ol{\e},
\end{equation} for $a_1,a_2 \in A_r$ and $\w, \e \in \O$. We assume that the faithful state $\t$ used to define the inner product \eqref{eq 4.3.12} is right invariant, i.e.,
\begin{equation}\label{eq 4.6.2}
\t(h \cdot b)=\t(\ve_r(h) \cdot b),
\end{equation} for $h \in H$ and $b \in B$. We have the following lemma.

\begin{lem}\label{lem 4.6.2}
For $\w, \e \in \O$ and $h \in H$,
\begin{equation}\label{eq 4.6.3}
\t g(\w \tens S(h) \cdot \ol{\e})=\t g (h \cdot \w \tens \ol{\e})
\end{equation} holds, where $g$ is as in Definition \ref{defn 4.3.36}. Thus $\<h \cdot \w, \e\>=\<\w, (S^2(h))^* \cdot \e\>.$
\end{lem}

\begin{proof}
The proof is essentially contained in \cite{MR2817646}. We compute
\[
\begin{aligned}
\t g(\w \tens S(h) \cdot \ol{\e})&=\t g(\w \tens s_r\ve_r(h^{(1)})S(h^{(2)}) \ol{\e}) \quad \eqref{eq 4.2.12}\\
&=\t(\ve_r(h^{(1)}) \cdot g(\w \tens S(h^{(2)}) \cdot \ol{\e})) \quad \eqref{eq 4.3.11}\\
&=\t(h^{(1)} \cdot g(\w \tens S(h^{(2)}) \cdot \ol{\e})) \quad \eqref{eq 4.6.2}\\
&=\t g(h_{(1)} \cdot \w \tens h^{(1)}_{(2)}S(h^{(2)}_{(2)}) \cdot \ol{\e}) \quad \eqref{eq 4.2.6}\\
&=\t g(h_{(1)} \cdot \w \tens \ve_l(h_{(2)}) \cdot \ol{\e})\\
&=\t g(t_l\ve_l(h_{(2)})h_{(1)} \cdot \w \tens \ol{\e})\\
&=\t g(h \cdot \w \tens \ol{\e}).
\end{aligned}
\] The last statement follows from the definition of $H$-action on $\ol{\O}$.
\end{proof}

Thus $H$ is represented by closable operators having a common dense domain. We denote the adjoint of $h \in H$ by $h^{\dagger}$ so that $h^{\dagger}=(S^2(h))^*$ on $\O$. From now on, let us allow a notational abuse of denoting by $h$ both the operator on $\O$ and its closure in $L^2(\O)$. At this point, we make an additional regularity assumption: 

\subsubsection*{\bf Assumption} Given $h \in H$, $\cD_h=\{\w \in \O \mid \sum_{0}^{\infty}\frac{\|h^n\w\|}{n!} < \infty\}$ is dense in $L^2(\O)$.

\begin{lem}\label{lem 4.6.3}
For $h \in H$ with $h=h^{\dagger}$ and $\w \in \cD_h$, define $U_h$ by \[U_h(\w)=\sum_n \frac{i^n}{n!}h^n\w,\] which is well-defined by the above $\mathbf{Assumption}$. Then $U_h$ extends to a unitary operator on $L^2(\O)$ denoted by $e^{ih}$.
\end{lem}

\begin{proof}
The result follows from the observations that for such an $h$, $\cD_h=\cD_{-h}$ and that $U_hU_{-h}=U_{-h}U_h=\id$.
\end{proof}



\begin{lem}\label{lem 4.6.4}
If the commutator $[h,d+d^*]$ extends to a bounded operator on $L^2(\O)$, then so does $[e^{ih},d+d^*]$.
\end{lem}

\begin{proof}
Observe that
\[
\begin{aligned}
e^{ih}(d+d^*)-(d+d^*)e^{ih}&=\int_{0}^{1}\frac{d}{ds}(e^{ish}(d+d^*)e^{i(1-s)h})ds\\
&=\int_{0}^{1}(ihe^{ish}(d+d^*)e^{i(1-s)h}-ie^{ish}(d+d^*)e^{i(1-s)h}h)ds\\
&=\int_{0}^{1}i(e^{ish}h(d+d^*)e^{i(1-s)h}-e^{ish}(d+d^*)he^{i(1-s)h})ds\\
&=\int_{0}^{1}i(e^{ish}[h,d+d^*]e^{i(1-s)h})ds.
\end{aligned}
\] 
As $e^{ith}$ is unitary, the integrand is bounded and the result follows.
\end{proof}

Combining Lemma \ref{lem 4.6.3} and Lemma \ref{lem 4.6.4}, we get the following proposition.

\begin{prop}\label{prop 4.6.5}
Let $\cA$ be the $\*$-algebra generated by operators of the form $ae^{i(h+h^{\dagger})}b$ with $a,b \in A_l$ and $h \in H_0$ in $B(L^2(\O))$. Then $(\cA,L^2(\O),d+d^*)$ forms a spectral triple. 
\end{prop}

\begin{proof}
We first observe that the representation of $A_l$ on $L^2(\O)$ is induced by restricting through $A_l \to B$, $a \mto s_l(a)\cdot 1_B$. For $b \in B$ and $\w \in \O$, \[\<b\w,b\w\>=\int_{\t}b\w \wdg \star(\w^*b^*)=\t(bg(\w,\w)b^*)\leq \|\w\|^2\t(bb^*),\] implying that left multiplication by $b$ extends to a bounded operator. In the above estimate, the inequality comes from the fact that $g(\w,\w)$ is positive in $B$. 

\medskip\noindent

Next, observing that $[s_l(a),d]$ is left multiplication by $d(s_l(a)\cdot 1_B)$, we prove that left multiplication by $d(s_l(a)\cdot 1_B)$ is bounded on $L^2(\O)$. Again we do this for $b \in B$. Note that
\[
\begin{aligned}
    \<db \wdg \w, db \wdg \w\>&=\<db \wdg \w, d(b\w)-bd\w\>\\
    &=\<db \wdg \w, d(b\w)\>-\<db \wdg \w, bd\w\>.
\end{aligned}    
\]
The first term in the above expression is estimated as follows:
\begin{equation}
    |\<db \wdg \w, d(b\w)\>|=|\<d^*(db \wdg \w),b\w\>| \leq \|d^*(db \wdg \w)\|\|bw\| \leq \text{ const. }\|\w\|,
\end{equation} where we used Cauchy-Schwarz inequality and that left multiplication by $b$ is bounded. The second term is estimated as follows:
\begin{equation}
    |\<db \wdg \w, bd\w\>| \leq \|db \wdg \w\|\|bdw\| \leq \text{ const. } \|d\w\|,
\end{equation} where we again used the boundedness of left multiplication by $b \in B$. Combining the two, we conclude that left multiplication by $db$ is bounded on $L^2(\O)$. Now observe that $s_l(a)\cdot 1_B$ is adjointable (see the discussion after Lemma \ref{lem 4.6.2}) with adjoint again an element of $B$ (use Eq. \eqref{eq 4.2.11} and compute: $(S^2(s_l(a)))^*=(Ss_r\ve_rs_l(a))^*=(s_l\ve_ls_r\ve_rs_l(a))^*=s_l((\ve_ls_r\ve_rs_l(a))^*)$) and for such an element the adjoint of $[d,b]$ is precisely $-[d^*,b^{\dagger}]$, $b^{\dagger}$ is the adjoint of $b$. Since adjoint of $[d,b]$ is bounded, we can conclude that $[b,d+d^*]$ extends to a bounded operator on $L^2(\O)$.

\medskip\noindent

Finally, $[s_l(a),d+d^*]$ extends to a bounded operator yields that $[h,d+d^*]$ too extends to a bounded operator for $h \in H_0$. The result follows from \ref{lem 4.6.4}. 
\end{proof}

If we assume that $\D_d$ has purely discrete spectrum then we get a spectral triple of compact type. We also note that \cite{MR3720811} computes the spectrum for the concrete examples. In our abstract setup, we propose a way of doing it generally. 
It would be interesting to know the
answer of the following:

\begin{que}\label{que 4.6.6}
If we assume an analogue of Relich's lemma $H^k \hookrightarrow H^{k+2}$ is compact then does it follow that $\D_d$ has purely discrete spectrum? See \cite{MR2277669} for the setup and more on abstract pseudo-differential calculi which has motivated this question.
\end{que} 
This would give a uniform way of proving that the Laplacian $\D_d$ has purely discrete spectrum in the setting of noncommutative differential calculi.

\subsection{Further examples}\label{pointer} As examples for our framework, we have mentioned \'etale groupoids, Hopf algebras, weak Hopf algebras and the Connes-Moscovici Hopf algebroid. There is another class of examples coming from Lie-Rinehart algebras and associated jet spaces; see \cite{MR2817646}. It would be interesting to 
know the answer of the following

\begin{que}\label{que 4.6.7}
Investigate if these examples fit into our framework. If so, what is the meaning of having a complex structure on a Lie-Rinehart algebra?
\end{que} 

As alluded to in the Introduction, Hopf algebroids appear naturally as symmetries of subfactors. A well-explored  rich source of subfactors is the theory of conformal nets (\citelist{\cite{MR3966864}\cite{MR1491122}\cite{MR3796433}}). Moreover, there has been some work (\cite{MR2585992}) on construction of spectral triples, i.e., noncommutative manifold structures, on ($C^*$-) algebras associated with conformal nets. Under certain conditions, the corresponding von Neumann algebras give rise to subfactors and as Hopf algebroids occur naturally as symmetries (\citelist{\cite{MR1913440}\cite{galois}}) of such subfactors, it is plausible that one may be able to construct noncommutative complex/Hermitian/K\"ahler structures on the module of one-forms of such spectral triples which will have equivariance with respect to some nontrivial and interesting Hopf algebroids. We have plans to explore this idea in our future work. 

\medskip\noindent

Furthermore, Hopf algebroid equivariance is not unexpected from another related viewpoint of \cite{MR1815717}, in which the author argues as follows. At the space level, one knows quantization of a Poisson manifold yields a noncommutative space and at the group level, one obtains quantum groups. If one agrees to view a groupoid as the joint generalization of a space and a group, then what should a quantization of a groupoid yield? The answer is a quantum groupoid or dually a Hopf algebroid, as shown in \cite{MR1815717}. In fact, it is shown therein that the algebra of differential operators on a Poisson manifold quantize to a canonical Hopf algebroid over the deformed function algebra. We also have plans to pursue this harmony between Hopf algebroid equivariance and deformation quantization.

\medskip

Let us end this article by noting that we only consider a fixed pseudogroup of local biholomorphisms of a given complex structure. It is indeed interesting and difficult to classify the complex structures, given a fixed set of local biholomorphisms. In the noncommutative setting, this amounts to the same thing as fixing a Hopf algebroid and investigating the moduli of all covariant complex structures. For quantum groups and spectral triples, these investigations have already been taken up, see \cites{MR2036597,MR1995873,MR2174423}. Our modest approach does not answer these questions and it would be very interesting to answer them.  







\section*{Acknowledgments}

The first author is grateful to Aritra Bhowmick for several discussions on foliations, to Yuri Kordyukov for answering many questions, to Edwin Beggs and Shahn Majid for helpful comments, and finally to R\'eamonn \'O Buachalla for countless discussions and his interest in this work. He also thanks the second author for answering many questions on complex geometry and the third author for introducing him to the theory of foliations. The second author is partially supported by J.C. Bose National Fellowship. The third author is partially supported by J.C. Bose National Fellowship and Research Grant awarded by D.S.T. (Govt. of India). 


\end{document}